\documentclass[a4paper, 11pt, english]{article}

\usepackage{inputenc}
\usepackage{amsmath, amsthm, amssymb}
\usepackage{mathtools}
\usepackage{color}
\usepackage{graphicx}
\usepackage[shortlabels]{enumitem}
\usepackage[toc]{appendix}

\usepackage{a4wide}
\usepackage{authblk}

\title{\textbf{\large Spherical Basis Functions in Hardy Spaces with Localization Constraints}}

\author[1]{C. Gerhards\thanks{christian.gerhards@geophysik.tu-freiberg.de}}
\author[1]{X. Huang\thanks{xinpeng.huang@geophysik.tu-freiberg.de}}

\affil[1]{
	TU Bergakademie Freiberg, Institute of Geophysics and Geoinformatics, \linebreak Gustav-Zeuner-Str. 12, 09599 Freiberg, Germany
}
\date{}
\usepackage[colorlinks=true, linkcolor=blue, citecolor=blue, pdfauthor={Author}]{hyperref}
\usepackage{nicefrac}
\renewcommand{\SS}{\mathbb{S}}
\renewcommand{\SS}{\mathbb{S}}

\newcommand{\R}{\mathbb{R}}

\newcommand{\f}{\mathbf{f}}
\newcommand{\g}{\mathbf{g}}

\renewcommand{\d}{\mathbf{d}}
\newcommand{\n}{\boldsymbol \eta}

\newcommand{\Ro}{\mathcal{R}}

\newcommand{\grads}{\nabla_\SS}
\newcommand{\gradss}[1]{\nabla_{\SS,#1}}

\newcommand{\ptm}{\tau_{\textnormal{ptm}}}
\newcommand{\mtp}{\tau_{\textnormal{mtp}}}

\newcommand{\gradS}{\nabla_\SS}

\newcommand{\supp}{\textnormal{supp}}

\newcommand{\zeroset}{\Sigma}

\newcommand{\eps}{\varepsilon} 

\renewcommand{\supp}{\textnormal{supp}} 
\newcommand{\id}{I}

\newcommand{\sobs}[1]{H^{#1}(\SS)}

\newcommand{\divfreezeroset}{\textnormal{V}_{\zeroset}}

\newcommand{\lts}{L^{2}(\SS)}

\newcommand{\sgp}{P_{stereo}}
\newcommand{\sgpx}{P_{stereo,\mathbf{p}}}

\newcommand{\ylk}{Y_{n,k}}
\newcommand{\ynk}{Y_{n,k}}

\newcommand{\Ctwo}{t}
\newcommand{\dualelement}{y}

\newcommand{\DeltaS}{\Delta_\SS}
\newcommand{\DeltaSS}[1]{\Delta_{\SS,#1}}
\newcommand\scalemath[2]{\scalebox{#1}{\mbox{\ensuremath{\displaystyle #2}}}}
\numberwithin{equation}{section}

\newtheorem{thm}{Theorem}[section]

\newtheorem{prop}[thm]{Proposition}
\newtheorem{cor}[thm]{Corollary}

\theoremstyle{definition}

\newtheorem{defi}[thm]{Definition}

\newtheorem{rem}[thm]{Remark}

\newenvironment{prfof}[1]{\noindent\textit{Proof of #1.}}
{\qed\\}

\newenvironment{abst}{\begin{center}\begin{minipage}[c]{0.9\textwidth} \footnotesize \textbf{Abstract.}}
		{\end{minipage}\\[5ex]\end{center}}

\begin{document}
\maketitle

\begin{abst}
Subspaces obtained by the orthogonal projection of locally supported square-integrable vector fields onto the Hardy spaces $H_+(\SS)$ and $H_-(\SS)$, respectively, play a role in various inverse potential field problems since they characterize the uniquely recoverable components of the underlying sources. Here, we consider approximation in these subspaces by a particular set of spherical basis functions. Error bounds are provided along with further considerations on norm-minimizing vector fields that satisfy the underlying localization constraint. The new aspect here is that the used spherical basis functions are themselves members of the subspaces under consideration.

\end{abst}

\section{Introduction}\label{sec:intro}

It is well-known that the space of square-integrable vector fields on the sphere can be decomposed as follows:
\begin{align}
	L^2(\SS,\R^3)=H_+(\SS)\oplus H_-(\SS) \oplus D(\SS),\label{eqn:hhdecomp}
\end{align}
where $H_+(\SS)$ and $H_-(\SS)$ denote the Hardy spaces of inner and outer harmonic gradients, respectively, and $D(\SS)$ the space of tangential divergence-free vector fields (see, e.g., \cite{atfeh10,freedengerhards12}, with roots in terms of vector spherical harmonic representations going back to \cite{edmonds57}). Applications to inverse magnetization problems and the separation of magnetic fields into contributions of internal and external origin are various and can be found, e.g., in \cite{backus96,gubbins11,gubjia22,mayermaier06,olsen10b,plattnersimons17,verles18}.     
 Concerning the inverse magnetization problem, knowing the magnetic field in the exterior of a magnetized sphere, only the $H_+(\SS)$-component of the underlying magnetization can be reconstructed uniquely. However, if it is known a priori that the magnetization is locally supported in some proper subdomain of the sphere, also the  $H_-(\SS)$-component is determined uniquely (which may help for the reconstruction of certain properties of the magnetization, e.g., the direction of an underlying inducing field or the susceptibility in the case of a known direction of the inducing field, e.g., \cite{gerhards16a,gerhards16b,verles19,lima13,vervelidou16}). Therefore, it is of interest to have a suitable setup that allows the computation of the $H_-(\SS)$-component based on knowledge of the $H_+(\SS)$-component under the assumption that the underlying vector field is locally supported. A first step is to investigate appropriate approximation methods in $H_+(\SS)$ and $H_-(\SS)$, respectively, under such localization constraints. 

The latter is the main motivation and content of the paper at hand. More precisely, let $L^2(\Sigma^c,\R^3)$ be the space of square-integrable vector fields that are locally supported in a subdomain $\Sigma^c\subset\SS$, and let $P_+(L^2(\Sigma^c,\R^3))$ and $P_-(L^2(\Sigma^c,\R^3))$ be the subspaces obtained by orthogonal projection onto the corresponding Hardy space $H_+(\SS)$ and $H_-(\SS)$, respectively. Here, we will investigate approximation in these projected subspaces. The specific vectorial Slepian functions from, e.g., \cite{plattnersimons14b,plattnersimons17} can address this to a certain extent. However, their construction requires joint optimization of the $H_+(\SS)$-, $H_-(\SS)$-, and $D(\SS)$-contributions and often focuses on a spectral representation. Furthermore, the computations have to be redone for every new subdomain of the sphere. Spherical basis functions, on the other hand, have the advantage that only their centers need to be adapted if the subdomain under consideration changes. Additionally, the construction as presented here allows for a separate treatment of the $H_+(\SS)$- and $H_-(\SS)$-contributions, respectively.

The course of this paper will be as follows: It is known from \cite{gerhuakeg23} that the orthogonal projection of $L^2(\Sigma^c,\R^3)$ onto the Hardy space $H_+(\SS)$ produces the subspace $P_+(L^2(\Sigma^c,\R^3))=B_+(D_{+,\Sigma})$, where $B_+$ denotes a particular vectorial linear operator that is expressible via layer potentials (cf. Section \ref{sec:prelhardy} for details and notations) and $D_{+,\Sigma}\subset L^2(\SS)$ denotes the scalar function space 
\begin{align}\label{eqn:dspmrep0}
	D_{\zeroset,+}=L^2(\zeroset^c)+(K+\tfrac{1}{2}I)V_\zeroset
\end{align}
(cf. Section \ref{sec:locrecap} for details and notations; for now, we let $K$ denote the double layer potential and $V_\Sigma$ the space of functions whose single layer potential is harmonic in $\Sigma$).  Thus, since the operator $B_+$ is bounded and invertible, the problem of finding adequate vectorial spherical basis functions for the approximation of the $H_+(\SS)$-component of locally supported vector fields reduces to finding adequate (scalar) spherical basis functions for approximation in the subspace $D_{+,\Sigma}$. The latter is, in fact, what we focus on in this paper (a similar statement holds true for the $H_-(\SS)$-component and a related subspace $D_{-,\Sigma}$). The structure of $D_{+,\Sigma}$ according to \eqref{eqn:dspmrep0} indicates that it suffices to investigate approximation in $L^2(\Sigma^c)$ and approximation in $V_\Sigma$. The former is discussed in Section \ref{sec:approxl2s}, the latter in Section \ref{sec:approxvsigma}. An approximation result for locally supported spherical basis functions similar to the one in Section \ref{sec:approxl2s} has already been obtained, e.g., in \cite{LeGia2017}. However, in their case, the support of the used spherical basis functions does not necessarily lie within $\Sigma^c$, while we want to enforce this support condition in order to guarantee that the functions used for approximation are members of $L^2(\Sigma^c)$ themselves. In a Euclidean setup, the latter has been achieved in \cite{towwen13}. We simply transfer their result to the sphere by applying the spherical techniques used, e.g., in \cite{legnarwar06,LeGia2017,legia10} and the stereographic projection. For the study of approximation in $V_\Sigma$ we rely on regularized fundamental solutions of the Laplace-Beltrami operator, as described, e.g., in \cite{freedenschreiner06,freedengerhards10,freedenschreiner09}. These spherical basis functions have the advantage that they are locally harmonic and naturally lead to a function system that is contained in $V_\Sigma$ itself. Both results are then joined in Section \ref{sec:approxhardy} to obtain the desired approximation property in $D_{+,\Sigma}$, which is stated in Theorem \ref{thm:densedsigma+}. As indicated before, the new aspect here is not the approximation property in general but that the investigated spherical basis functions lead to vectorial functions that are themselves members of the projected subspace $P_+(L^2(\Sigma^c,\R^3))$ of $H_+(\SS)$. Sections \ref{sec:minnorm} and \ref{sec:minnormsbf} additionally provide some considerations on locally supported vector fields of minimum norm and their representation in the given context. Sections \ref{sec:back1} and \ref{sec:back2} gather the required notations and background material, some numerical illustration is provided in Section \ref{sec:numerics}.

\section{General preliminaries}\label{sec:back1}

We define some necessary notations required throughout the course of this paper and gather some known results on spherical basis function interpolation, cubature on the sphere, and the fundamental solution for the Laplace-Beltrami operator on the sphere.

\subsection{Sobolev spaces on the sphere}\label{sec:sobspace}

Let $\SS=\{x\in\R^3:|x|=1\}$ denote the unit sphere. We restrict ourselves to the three-dimensional case to simplify some considerations, especially in Sections \ref{sec:reggreen} and \ref{sec:approxvsigma}, although they should also hold in higher dimensional setups. Let $\{Y_{n,k}\}_{n\in\mathbb{N},k=-n,\ldots,n}$ denote an orthonormal set of real spherical harmonics of order $n$ and degree $k$ (each spherical harmonic of degree $n$ and  order $k$ is the restriction of a homogeneous, harmonic polynomial of degree $n$ to the sphere; details can be found, e.g., in \cite{freedenschreiner09}). For an infinitely often differentiable function $f:\SS\to\R$ of class $C^\infty(\SS)$, one can write $f=\sum_{n=0}^\infty\sum_{k=-n}^n\hat{f}_{n,k}Y_{n,k}$ with Fourier coefficients $\hat{f}_{n,k}=\int_\SS f(y)Y_{n,k}(y)d\omega(y)$, where $\omega$ denotes the surface measure on the sphere. When dealing with such Fourier coefficients, we often say that we are arguing in spectral domain. The Sobolev space $H^s(\SS)$ of smoothness $s\in\R$ is defined as the closure of $C^\infty(\SS)$ with respect to the norm $\|\cdot\|_{H^s(\SS)}$ induced by the inner product
\begin{align}\label{eqn:innerprodhs}
	\langle f,g\rangle_{H^s(\SS)}=\sum_{n=0}^\infty\sum_{k=-n}^n\left(n+\frac{1}{2}\right)^{2s}\hat{f}_{n,k}\,\hat{g}_{n,k},
\end{align}
i.e., $H^s(\SS)=\overline{C^\infty(\SS)}^{\|\cdot\|_{H^s(\SS)}}$. In particular, it holds that $H^0(\SS)=L^2(\SS)$, the space of square-integrable functions, and that $C^0(\SS)\subset H^s(\SS)$ for $s>1$, where $C^0(\SS)$ denotes the set of continuous functions on the sphere. By $\|\cdot\|_{C^0(\SS)}$ we denote the supremum norm. At some occasions, we use the notation $L^2(\SS)/\langle1\rangle$ to indicate the space of functions $f$ in $L^2(\SS)$ with vanishing mean $\langle f,1\rangle_{L^2(\SS)}=0$. 

We denote by $\grads$ the (surface) gradient on the sphere. For $f$ in $H^1(\SS)$, its gradient $\grads f$ is well-defined and lies in the space $L^2(\SS,\R^3)$ of square-integrable vector fields, which is equipped with the inner product
\begin{align}
		\langle \f,\g\rangle_{L^2(\SS)}=\int_\SS \f(y)\cdot\g(y)\,d\omega(y)
\end{align}
and the corresponding norm $\|\cdot\|_{L^2(\SS,\R^3)}$. Typically, we denote vector fields $\f,\g:\SS\to\R^3$ by bold-face letters. The (surface) divergence operator $\grads\cdot$ is defined as the negative adjoint $-\grads^*$ (with respect to the $L^2$-inner product) and we define the Laplace-Beltrami operator via $\Delta_\SS=\grads\cdot\grads=-\grads^*\grads$. 

If $\Sigma\subset\SS$ is a Lipschitz domain with connected boundary, then the space of square-integrable locally supported vector fields on $\Sigma$ is denoted by $L^2(\Sigma,\R^3)$, naturally embedded in $L^2(\SS,\R^3)$ via extension by zero. Furthermore, we define $H^s(\Sigma)$ as the closure of $C_0^\infty(\Sigma)=\{f\in C^\infty(\SS):\supp(f)\subset\Sigma\}$ with respect to the $\|\cdot\|_{H^s(\SS)}$-norm. Thus, $H^s(\Sigma)$ is naturally embedded in $H^s(\SS)$.

\begin{rem}\label{rem:spectrepdeltas}
	For $f$ in  $H^2(\SS)$ it holds that $(\Delta_\SS f)^\wedge_{n,k}=-n(n+1)\hat{f}_{n,k}$ and, consequently, that $((-\Delta_\SS+\frac{1}{4}) f)^\wedge_{n,k}=\big(n+\frac{1}{2})^2\hat{f}_{n,k}$. Thus, one can (formally) see from \eqref{eqn:innerprodhs} that a function $f$ is in $H^s(\SS)$ if and only if $(-\Delta_\SS+\frac{1}{4})^{\frac{s}{2}}f$ is in $\lts$.
\end{rem}

A function $f:\SS\to\R$ is called $\mathbf{p}$-zonal (for some fixed $\mathbf{p}\in \SS$) if there exists a function $F:[-1,1]\to\R$ such that $f(x)=F(x\cdot \mathbf{p})$ for all $x\in\SS$. In that case, we define $\Ro_{x}f(y)=f(z)$, for $z\in\SS$ chosen such that $x\cdot y=z\cdot \mathbf{p}$. This allows the definition of the convolution $f*g$ of two functions $f$ and $g$ via
\begin{equation}\label{eqn:conv}
	f\ast g (x)= \int_{\SS} \Ro_{x}f(y)\,g(y)d\omega(y),\quad x\in\SS.
\end{equation}
For $F:[-1,1]\to\R$, the notation $F*g$ is meant in the sense $F\ast g (x)= \int_{\SS} F(x\cdot y)\,g(y)d\omega(y)$, naturally relating to the definition \eqref{eqn:conv}.

\begin{prop}\label{thm:smoothnessofconvolution}
	If $f$ is a $\mathbf{p}$-zonal function in $\sobs{s}$ and $g$ is in $\sobs{k}$, for some $s,k\in\R$, then
\begin{equation}
	\|f\ast g\|_{\sobs{s+k+\nicefrac{1}{2}}} \leq C  \| f\|_{\sobs{s}} \| g\|_{\sobs{k}},
\end{equation}
where $C>0$ is a constant depending  on $s$ and $k$.
\end{prop}

A proof of the proposition above is provided in Appendix \ref{sec:prfprop21}. Furthermore, the following estimate will be helpful.

\begin{prop}{\cite[Theorem 24]{Coulhon2001}}\label{thm:smoothnessofmultiplication}
Let $s>1$. Then $\sobs{s}$ is a Banach algebra, i.e., there exists a constant $C>0$, depending  on $s$, such that for any $f,g$ in $\sobs{s}$, the following inequality holds
\begin{equation}
	\|fg\|_{\sobs{s}}\leq C\|f\|_{\sobs{s}}\|g\|_{\sobs{s}}.
\end{equation}
\end{prop}

To conclude this section, we want to mention that occasionally it is convenient to rely on known results in the Euclidean setup by using the stereographic projection. Given a fixed pole $\mathbf{p}\in\SS$, let $\sgp:\SS\setminus\{\mathbf{p}\}\to\R^2$ be the stereographic projection as defined, e.g., in \cite[Def. 2.34]{freedengerhards12}. 

Typically we do not indicate the pole explicitly in the notation of the stereographic projection. If it should be required for clarity, we write $\sgpx$ instead of $\sgp$. The stereographic projection of a set $\Sigma\subset \SS\setminus\{\mathbf{p}\}$ is given by $\sgp(\Sigma)=\{\sgp(x):x\in\Sigma\}$. Furthermore, the stereographic projection of a function $f:\SS\setminus\{\mathbf{p}\}\to\R$ is defined via
\begin{align}
	\sgp f:\R^2 \to \R,\qquad y \mapsto f(\sgp^{-1}(y)).
\end{align}
The following are some useful estimates involving the stereographic projection. 

\begin{prop}\label{prop:propertyofstereographicprojection}
	Let $\mathbf{p}\in\SS$ be fixed and $\Sigma\subset\SS$ be a Lipschitz domain with connected boundary whose closure satisfies $\overline{\Sigma}\subset\SS\setminus\{\mathbf{p}\}$. Then there exist constants $C>c>0$ such that, for any $x,y\in\Sigma$,
	\begin{align}
		c|x-y|&\leq |\sgp(x)-\sgp(y)|\leq C |x-y|,\label{eqn:upboundstereopt}
		\\c\,\omega(\Sigma)&\leq \mu(\sgp(\Sigma))\leq C\,\omega(\Sigma),
	\end{align}
	where $\mu$ denotes the Lebesgue measure in $\R^2$ and $\omega$ the surface measure on the sphere $\SS$.
	
	Let $s\geq 0$, then there exist constants $C\geq c>0$ (possibly different from above), depending  on $s$, such that for any $f\in H^s(\Sigma)$ the following inequalities hold true:
	\begin{align}
		c\|f\|_{H^{s}(\SS)} \leq \|\sgp f\|_{H^{s}(\R^2)}\leq C\|f\|_{H^{s}(\SS)}.
	\end{align}
\end{prop}

Throughout the course of this article, we allow all appearing constants to depend on the subdomain $\Sigma\subset\SS$ without explicitly mentioning this. In cases where the constant does not depend on the subdomain, this should become clear from the context.

\subsection{Cubature rules on the sphere}\label{sec:cub}

Given a set of points $X=\{x_1,\ldots,x_N\}\subset\SS$ and weights $w_1,\ldots,w_N\in\R$, then the corresponding cubature rule $Q_X$ takes the form
\begin{equation}
	Q_{X}f=\sum_{i=1}^{N} w_i f(x_i).
\end{equation}
We say that $Q_X$ has polynomial precision of degree $L$ if $Q_Xf=\int_{\SS} f(y)d\omega(y)$ for all spherical polynomials of degree at most $L$, i.e., for all functions of the form $f=\sum_{n=0}^L\sum_{k=-n}^n\hat{f}_{n,k}Y_{n,k}$. If all weights are positive, i.e., if $w_i>0$ for all $i=1,,\ldots,N$, then we call $Q_X$ a cubature rule with positive weights. The spatial distribution of the points in $X$ is paramount for the existence of cubature rules with a prescribed polynomial precision and positive weights. An important parameter that provides information on the spatial distribution is the mesh width of $X$ relative to some subset $\Sigma\subset\SS$, which is defined by 
\begin{align}
	h_{X, \Sigma}=\sup_{y\in\Sigma}\min_{x_i\in X\cap \Sigma}\|y-x_i\|.
\end{align} 
If $\Sigma=\SS$, we typically abbreviate $h_{X, \Sigma}$ simply by $h_{X}$.
The separation distance is defined by 
\begin{align}
	q_{X,\Sigma}=\frac{1}{2}\min_{x_i,x_j\in X\cap\Sigma\atop x_i\not=x_j}\|x_i-x_j\|,
\end{align} 
again, we drop the index $\Sigma$ if $\Sigma=\SS$.

\begin{thm}{(Existence of cubature rules \cite[Theorem 6]{hesse15})}\label{thm:cubature2}
	There exists a constant $C>0$ such that following holds true: For every pointset $X\subset\SS$ and for every $L\in\mathbb{N}$ satisfying
	\begin{equation}
		h_{X}\leq \frac{C}{L},\label{eqn:meshnormcond}
	\end{equation}
	there exist positive weights $w_1, \ldots ,w_N>0$ such that the corresponding cubature rule $Q_X$ has polynomial precision of degree $L$.
\end{thm}

The existence of cubature rules with polynomial precision of degree $L$ and positive weights then guarantees the following error estimate for our later considerations in Section \ref{sec:approxvsigma}.  

\begin{thm}{(Error estimate for cubature rules \cite[Theorem 11]{hesse15})}\label{thm:cubature1}
	Let $s>1$ and $X\subset\SS$. Then there exist a constant $C>0$ that depends on $s$ such that, for every $L\in\mathbb{N}$ and for every cubature rule $Q_X$ with polynomial precision of degree $L$ and positive weights, the worst-case error in $H^s(\SS)$ can be estimated by
	\begin{equation}
		\sup_{f\in\sobs{s},\|f\|_{\sobs{s}}\leq 1}\left|Q_Xf-\int_{\SS} f(y)d\omega(y)\right|\leq \frac{C}{L^s}.
	\end{equation}
\end{thm}

\subsection{Spherical basis functions and interpolation}\label{sec:spherbasis}

In this section, we briefly recapitulate some definitions and results for spherical basis functions derived from classical radial basis functions, as indicated, e.g., in \cite{legnarwar06,LeGia2017,legia10}.

\begin{defi}\label{def:dilatioylkernels}
	Let $\psi:[0,\infty)\to\R$ be compactly supported with $\supp(\psi)\subset[0,1]$ such that $\psi(|\cdot|):\R^3\to\R$ is a positive definite radial basis function. Then we define $\Psi:\SS\times\SS\to\R$ via $\Psi(x,y)=\psi(|x-y|)$ and introduce scaled versions, for $\delta>0$, via
	\begin{align}
		\Psi_\delta(x,y)=\frac{1}{\delta^2}\Psi\left(\frac{x}{\delta},\frac{y}{\delta}\right),\qquad x,y\in\SS.
	\end{align}
	We call these restrictions to the sphere \emph{spherical basis functions}. Often, we use the notation $\Psi_x$ and $\Psi_{\delta,x}$ to more concisely express the functions $\Psi(x,\cdot)$ and $\Psi_\delta(x,\cdot)$, respectively. 
	
	Further, observing that $\Psi(x,y)=\sum_{n=0}^{\infty}\sum_{k=-n}^{n}\hat{\Psi}_n\,Y_{n,k}(x)Y_{n,k}(y)$ for adequate coefficients $\hat{\Psi}_n\in[0,\infty)$, we define a norm $\|\cdot\|_\Psi$ via
	\begin{align}\label{eqn:nativenorm}
		\|f\|_{\Psi}^2 = \sum_{n=0}^\infty\sum_{k=-n}^{n}\frac{ \hat{f}_{n,k}^2 }{\hat{\Psi}_n},\qquad f\in C^\infty(\SS).
	\end{align}
	The native space corresponding to $\Psi$ is then given as the closure of $C^\infty(\SS)$ with respect to this norm $\|\cdot\|_\Psi$. Analogous definitions hold for $\Psi_\delta$, where the involved coefficients are denoted by $\left(\Psi_\delta\right)^\wedge_n$.
\end{defi}

A popular choice for the underlying $\psi$ in Definition \ref{def:dilatioylkernels} that we will use in our numerical examples is, e.g., the Wendland function (cf. \cite{wendland95,wendland05})
\begin{align}
	\psi(r)&=(1-r)_+^4(4r+1),\label{eqn:wf1}
\end{align}
which corresponds to the native space $H^{\frac{5}{2}}(\SS)$, although various other choices are equally valid. A closed-form representation of its Legendre coefficients is provided in \cite{hubjae21}. 

\begin{prop}{\cite[Thm. 4.7]{hubjae21}} Let $\Psi_\delta$ be given as in Definition \ref{def:dilatioylkernels}, and the underlying $\psi$ be given as in \eqref{eqn:wf1}. Then, $\Psi_\delta$ has the following spherical harmonic expansion
\begin{align}
	\Psi_\delta(x,y)=\sum_{n=0}^{\infty}  \sum_{k=-n}^{n}\hat{\Psi}^\delta_n\,\ynk(x)\ynk(y)=\sum_{n=0}^{\infty}  \frac{2n+1}{4\pi}\hat{\Psi}^\delta_n\, P_n(x\cdot y),
\end{align}
with coefficients
\begin{align}
\hat{\Psi}^\delta_n=\frac{\pi}{7}\,{}_3F_2\left(-n,n+1,\frac{5}{2};4,\frac{9}{2};\frac{\delta^2}{4}\right),
\end{align}
where ${}_3F_2$ denotes the general hypergeometric function.
\end{prop}

Next, we state some known properties for interpolation via spherical basis functions.

\begin{prop}{\cite[Theorem 3.3]{legnarwar06}}\label{prop:zerolemma} Let $X\subset \SS$ have meshwidth $h_X$, and let $s>1$. Then there exists a constant $C>0$, depending  on $s$, such that, for all $f$ in $\sobs{s}$ with $f|_X=0$,  it holds
	\begin{align}
		\|f\|_{L^2(\SS)}\leq C h_X^s\|f\|_{\sobs{s}}.
	\end{align}
\end{prop}

We do not actually need the spherical estimate from Proposition \ref{prop:zerolemma} in the later course of this paper, it is stated mainly for illustration purposes. For the proof of Theorem \ref{thm:l2sigmaapprox}, we will infact use the stereogrpahic projection in order to retreat to the Euclidean setup and rely on related results from \cite{narwarwen05,towwen13}. 

\begin{prop}\label{prop:projectioninnativespace}
	Let $\Psi$ be as in Definition \ref{def:dilatioylkernels} with native space $H^{s}(\SS)$, for some $s>1$. Moreover, let $X\subset \SS$ and denote by $I_{X,\Psi}:H^{s}(\SS)\to \textnormal{span}\{\Psi_{x}: x\in X\}$ the interpolation operator that is characterized by $I_{X,\Psi}(f)|_X=f|_X$. Then, it holds
	\begin{align*}
		\|f\|_{\Psi}^2 =\|f-I_{X,\Psi}(f)\|_{\Psi}^2+\|I_{X,\Psi}(f)\|_{\Psi}^2.
	\end{align*}
\end{prop}

Some useful results on the spectral behaviour of the scaled spherical basis functions and on the boundedness of the corresponding interpolation operator are the following.

\begin{prop}{\cite[Theorem 6.2]{legia10}}\label{prop:spectraofdilatedkernel}
		Let $\Psi$ be as in Definition \ref{def:dilatioylkernels} where $\psi(|\cdot|)$ is a radial basis funtion in $\R^3$ with native space $H^{s+\nicefrac{1}{2}}(\R^3)$ for some $s>1$, then $\Psi$ is a spherical basis function with native space $H^{s}(\SS)$. Moreover, there exist constants $C\geq c>0$, depending  on $s$, such that  
	\begin{equation}\label{eqn:spectradecay}
		c(1+\delta n)^{-2s} \leq\left({\Psi}_{\delta}\right)^\wedge_n \leq C(1+\delta n)^{-2s}.
	\end{equation}
	As a result, there exists constants $\bar{C}\geq \bar{c}>0$, depending  on $s$, such that, for any $f$ in $H^{s}(\SS)$, 
	\begin{equation}
		\bar{c}\|f\|_{\Psi_{\delta}} \leq \|f\|_{\sobs{s}}\leq \bar{C} \delta^{-s} \|f\|_{\Psi_{\delta}}.
	\end{equation}
\end{prop}

\begin{prop}\label{Linftyboundedness}
	Let $\Psi$ satisfies the condition in Proposition \ref{prop:spectraofdilatedkernel}. Given $X\subset \SS$, let $c>1$ be such that $h_X< c q_X$, and let the scaling factor $\delta$ satisfy $\nu q_X<\delta< c\nu q_X$ for some $\nu\geq 1$. Then there exists a constant $C>0$, depending on $c$, $\nu$, and $s$, such that it holds for any $f$ in $H^{s}(\SS)$ that
	\begin{equation}\label{eqn:linftyestsph}
		\|I_{X, \Psi_{\delta}}(f) \|_{L^{\infty}(\SS)} \leq C \|f \|_{L^{\infty}(\SS)}.
	\end{equation}
\end{prop}

Proposition \ref{Linftyboundedness} is a simplified spherical version of  \cite[Theorem 3.5]{towwen13}. A proof can be found the Appendix \ref{app:cinftest}.

\subsection{Regularized fundamental solution}\label{sec:reggreen}

The function $G(\DeltaS;\cdot,\cdot):\{(x,y)\in\SS\times\SS:x\not=y\}\to\R$ given by
\begin{align}
	G(\DeltaS;x,y)=\frac{1}{4\pi}\ln(1-x\cdot y)+\frac{1}{4\pi}(1-\ln2).
\end{align}
is called fundamental solution with respect to $\DeltaS$ and is twice continuously differentiable on its domain. It satisfies 
\begin{equation}
	\DeltaSS{y} G(\DeltaS;x,y) = -\frac{1}{4\pi},\quad x\not=y,
\end{equation}
and it has the Fourier expansion, for $x\not=y$,
\begin{align}\label{eqn:seriesrepG}
	G(\DeltaS;x,y)=\sum_{n=1}^{\infty}\sum_{k=-n}^{n}\frac{1}{-n(n+1)}Y_{n,k}(x)Y_{n,k}(y)=\frac{1}{4\pi}\sum_{n=1}^{\infty}\frac{2n+1}{- n(n+1)}P_n(x\cdot y),
\end{align}
where $P_n$ denotes the Legendre polynomial of degree $n$ (see, e.g., \cite{freedenschreiner09}). Furthermore, Green's third formula yields for $f$ in $H^{3+\kappa}(\SS)$, with $\kappa>0$, that 
\begin{align}
	f(x)=\frac{1}{4\pi}\int_{\SS}f(y)d\omega(y) +\int_{\SS} G(\DeltaS;x,y) \DeltaSS{y}f(y)d\omega(y),\quad x\in\SS.
\end{align}
From \eqref{eqn:seriesrepG}, we can easily check, for fixed $x\in\SS$, that $G(\DeltaS;x,\cdot)$ lies in the Sobolev space $H^{s}(\SS)$ for any $s<1$ but not for $s\geq 1$. Regularization of the fundamental solution around its singularity leads to higher smoothness and is considered here for later use. A once continuously differentiable regularization will suffice for our purposes and a particular choice is presented in Definition \ref{def:regularizedgreenfunction}, although various other choices are possible as well. The concept is described and applied in more detail, e.g., in \cite{freedenschreiner06,freedengerhards10,freedengerhards12,freedenschreiner09}. 

\begin{defi}\label{def:regularizedgreenfunction}
For fixed $\rho\in(0,2)$, we define the once continuously differentiable regularized fundamental solution $G^{\rho}(\DeltaS;\cdot,\cdot):\SS\times\SS\to\R$ via
\begin{align}\label{eqn:reggreen}
	G^{\rho}(\DeltaS;x,y)=
	\begin{cases}
	\frac{1}{4\pi} \ln(1-x\cdot y)+\frac{1}{4\pi}(1-\ln 2),& x\cdot y\leq1-\rho, \\
	\frac{1-x\cdot y}{4\pi\rho} +\frac{1}{4\pi}(\ln \rho-\ln 2), &x\cdot y>1-\rho .
	\end{cases}
\end{align}
A later on useful notation will be the following, with $x,\bar{x}\in\SS$ fixed,
\begin{align}\label{eqn:grxx}
	G^\rho_{x,\bar{x}}=G^{\rho}(\DeltaS;x,\cdot)-G^{\rho}(\DeltaS;\bar{x},\cdot).
\end{align}

\end{defi}

\begin{rem}\label{rem:relrhor}
It is directly obvious that $G^{\rho}(\DeltaS:x,\cdot)$ and $ G(\DeltaS:x,\cdot)$ only differ on the spherical cap \begin{align}
	\mathcal{C}_\rho(x)=\{y\in\SS:x\cdot y> 1-\rho\}
\end{align}
with center $x$ and polar radius $\rho\in(0,2)$. Thus, $\DeltaS G^\rho_{x,\bar{x}}$ is locally supported in $\mathcal{C}_\rho(x)\cup \mathcal{C}_\rho(\bar{x})$. One should note that the spherical basis functions in Section \ref{sec:spherbasis} as well the mesh width $h_X$  in Section \ref{sec:cub} are defined with respect to the Euclidean distance, i.e., they relate to the spherical ball 
\begin{align}
	\mathcal{B}_r(x)\cap\SS=\{y\in\SS:|x-y|< r\}
\end{align}
with center $x$ and (Euclidean) radius $r>0$. The distinction between these two notions of spherical caps and spherical balls is solely made to pay tribute to the origins of the two different setups, not because they are conceptually different. In order for $\mathcal{C}_\rho(x)$ and $\mathcal{B}_r(x)\cap\SS$ to cover the same area, one needs to choose the radii to be related via $r=(2\rho)^{\frac{1}{2}}$. 
\end{rem}

As already mentioned, the regularization \eqref{eqn:reggreen} is only one of many possible choices. It has been chosen solely so that we can rely on some already known explict computations that lead to the following properties.

\begin{prop}\label{prop:regularizedgreenfunction}
Let the regularized fundamental solution $G^\rho(\DeltaS;\cdot,\cdot)$ be given as in Definition \ref{def:regularizedgreenfunction}.
\begin{enumerate}
\item It holds that 
\begin{align}
	\sup_{x\in\SS}\| G^{\rho}(\DeltaS:x,\cdot)\|_{\sobs{2}}
	\sim\sup_{x\in\SS}\| \DeltaS G^{\rho}(\DeltaS:x,\cdot)\|_{L^2(\SS)}\sim \rho^{-\tfrac{1}{2}},	
\end{align}
where the notation $f_\rho\sim g_\rho$ is meant in the sense that there exist constants $C\geq c>0$ such that $c\, g_\rho\leq f_\rho\leq C\,g_\rho$ for all sufficiently small $\rho>0$.
	\item For $f$ of class $C^{0}(\SS)$, it holds that (cf. \cite[Remark 8.19]{freedengerhards12})
\begin{align}
		\sup_{x\in\SS}\left| \int_{\SS}\left(G^{\rho}(\DeltaS:x,y)	-G(\DeltaS;x,y)\right)	f(y)d\omega(y)\right| &\leq C\rho\ln(\rho)\|f\|_{C^0(\SS)},
\\
	\sup_{x\in\SS}\left|  \int_{\SS}\left(\gradss{x} G^{\rho}(\DeltaS:x,y)-\gradss{x} G(\DeltaS;x,y)\right)	f(y)d\omega(y)\right| &\leq C\rho^{\frac{1}{2}}\|f\|_{C^0(\SS)},
\end{align}
with some constant $C>0$. 
\item $G^{\rho}(\DeltaS;\cdot,\cdot)$ has the following spherical harmonic expansion (cf. \cite[Lemma 4.7]{freedenschreiner09})
\begin{align}
	G^{\rho}(\DeltaS;x,y)=\sum_{n=0}^{\infty}  \sum_{k=-n}^{n}\hat{G}^\rho_n\,\ynk(x)\ynk(y)=\sum_{n=0}^{\infty}  \frac{2n+1}{4\pi}\hat{G}^\rho_n\, P_n(x\cdot y),
\end{align}
with coefficients
\begin{align}
\hat{G}^\rho_0&=\frac{1}{4}\rho,\label{eqn:ghn1}\\
\hat{G}^\rho_1&=-\frac{1}{2}+\frac{1}{4}\rho -\frac{1}{24}\rho^2,\label{eqn:ghn2} \\
\hat{G}^\rho_n&=\frac{P_{n+1}(1-\rho)-P_{n-1}(1-\rho)}{2n(n+1)(2n+1)}-\frac{2-\rho}{2n(n+1)}P_n(1-\rho)\label{eqn:ghn3}
 \\&\quad+\frac{1}{2\rho}\frac{P_n(1-\rho)-P_{n+2}(1-\rho)}{(2n+1)(2n+3)}-\frac{1}{2\rho}\frac{P_{n-2}(1-\rho)-P_{n}(1-\rho)}{(2n+1)(2n-1)},\quad n\geq2. \nonumber
\end{align}
\end{enumerate}
\end{prop}

\section{Preliminaries on Hardy spaces and localization constraints}\label{sec:back2}

In order to state the desired relations between $H_+(\SS)$ and $H_-(\SS)$ under localization constraints, we very briefly recapitulate classical layer potentials and the definition of Hardy spaces. Subsequently, we state the subspaces of $H_+(\SS)$ and $H_-(\SS)$ in which we want to perform approximation. This forms the foundation for most considerations in the later course of this paper. 

\subsection{Layer potentials and Hardy spaces}\label{sec:prelhardy}

For \( f \in \lts \) we define the (boundary) single layer potential $S \colon \lts \to H^1(\SS)$ via
\begin{equation*}
	Sf(x)
	=-\frac{1}{4\pi}\int_{\SS} \frac{f(y)}{|x-y|} \ d \omega(y),
	\quad x\in  \SS.
\end{equation*}
$S$ is bounded and invertible. Furthermore, we  define the operators $K - \tfrac{1}{2} \id \colon \lts / \langle 1 \rangle \to \lts / \langle 1 \rangle$ and $K + \tfrac{1}{2} \id \colon \lts \to \lts$, where \( \id \) is the identity operator and \( K \) is the (boundary) double layer potential
\begin{equation*}
	Kf(x)= 	-\textnormal{p.v.}\, \frac{1}{4\pi} \, \int_{\SS} 	\frac{\n(y)\cdot (x-y)}{|x-y|^{3} } f(y) \ d \omega(y), \quad x\in  \SS,
\end{equation*}
with $\n(y)=y$ simply denoting the unit normal vector to the sphere at the location $y\in\SS$. Both operators $K + \tfrac{1}{2} \id$ and $K - \tfrac{1}{2} \id$ are invertible and self-adjoint. The operator \( K + \tfrac{1}{2} \id \) preserves constant functions while the operator \( K - \tfrac{1}{2}I \) annihilates them. Latter is the reason why we formally define it only on the space $\lts / \langle 1\rangle$, although we will not require this particular distinction throughout most parts of the paper. For more information on (boundary) layer potentials we refer the reader, e.g., to  \cite{fabmenmit98,ver84} and references therein. Here, we merely use them to define the operators $B_{+}\colon \lts / \langle 1 \rangle\to  L^2(\SS,\R^3)$ and $B_{-} \colon \lts\to L^2(\SS,\R^3)$ given by 
\begin{align}\label{cross_s}
		B_{+}=\n \ \big(K-\tfrac{1}{2} \id \big)+\grads S,&&
		B_{-}=\n \ \big(K+\tfrac{1}{2} I\big)+\grads S.
\end{align}
The Hardy spaces are then defined as the ranges of the corresponding operators $B_+$ and $B_-$, i.e.,
\begin{align}
	H_+(\SS)=B_+( \lts / \langle 1 \rangle),&& H_-(\SS)=B_-( \lts).
\end{align} 
This notation will be useful for the purposes of the paper at hand; in particular, it allows to reduce considerations in the vectorial Hardy spaces to considerations in the scalar $\lts$ space. For the statement that these definitions are equivalent to the more classical definition of Hardy spaces via non-tangential limits of harmonic gradients, we refer the reader to \cite{bargerkeg20}. There, one can also find the statement of the Hardy-Hodge decomposition
\begin{align}
	L^2(\SS,\R^3)=H_+(\SS)\oplus H_-(\SS) \oplus D(\SS),\label{eqn:hhdecomp2}
\end{align}  
where $D(\SS)=\{\f\in L^2(\SS,\R^3):\f\textnormal{ is tangential and }(\grads S)^*\f=0\}$ denotes the space of tangential, divergence-free vector fields on the sphere. Note again that $B_+$ can also be considered as acting on the full $L^2(\SS)$ where it annihilates constant functions. Furthermore, it should be mentioned that the operators $B_+$, $B_-$ closely resemble the operators $\tilde{o}^{(1)}$, $\tilde{o}^{(2)}$ in \cite{freedengerhards12,freedenschreiner09,gerhards16a} (however, the notation of $B_+$ and $B_-$ more naturally allows considerations on Lipschitz surfaces, e.g., in \cite{bargerkeg20}, while the notation of $\tilde{o}^{(1)}$ and $\tilde{o}^{(2)}$ is rather specific to the sphere).

\begin{rem}
	Some basic calculations show that, for $f\in L^2(\SS)$, it holds $(Sf)^\wedge_{n,k}=-\frac{1}{2n+1}\hat{f}_{n,k}$ (see, e.g., \cite[Sect. 2.5.2]{freedengerhards12}). Analogously, it holds $(Kf)^\wedge_{n,k}=\frac{1}{4n+2}\hat{f}_{n,k}$. Thus, the application of the single and double layer potential can be easily computed in spectral domain.
\end{rem}

\subsection{Localization constraints}\label{sec:locrecap}

\begin{defi}
For the remainder of this paper, we assume that $\Sigma\subset \SS$, with $\overline{\Sigma}\not=\SS$, is a Lipschitz domain with connected boundary. Let $\Sigma^c=\SS\setminus \overline{\Sigma}$ be the open complement of $\Sigma$, which again is a Lipschitz domain with connected boundary. By $P_+:L^2(\SS,\R^3)\to H_+(\SS)$ and $P_-:L^2(\SS,\R^3)\to H_-(\SS)$ we denote the orthogonal projections onto $H_+(\SS)$ and $H_-(\SS)$, respectively. For any vector field $\f$ in $L^2(\SS,\R^3)$, we abbreviate its Hardy components by $\f_+=P_+(\f)$ and $\f_-=P_-(\f)$. Additionally, $P_\Sigma:L^2(\SS)\to L^2(\Sigma)$ denotes the projection onto $L^2(\Sigma)$, and we define the auxiliary space 
\begin{align}\label{eqn:vs}
	V_\Sigma=\{f\in L^2(\SS):\grads\cdot(\grads Sf)|_\Sigma=0 \textnormal{ as distribution on }\Sigma\}
\end{align}
of functions whose single layer potential is harmonic in $\Sigma$.
\end{defi}

In order to comply with the notation in \cite{gerhuakeg23}, we will implicitly assume that $\supp(\mathbf{f})\subset\Sigma^c$ if we say that a vector field is locally supported, i.e., we assume that $\f$ is in $L^2(\Sigma^c,\R^3)$. The upcoming theorem (which is a collection of results from \cite{gerhuakeg23}) states the desired unique relation between the $H_+(\SS)$-contribution $\f_+$ and the $H_-(\SS)$-contribution $\f_-$ of such a locally supported vector field $\f$. In particular, it characterizes the relevant subspaces $P_+(L^2(\Sigma^c,\R^3))\subset H_+(\SS)$ and $P_-(L^2(\Sigma^c,\R^3))\subset H_-(\SS)$. 

\begin{thm}\label{thm:tauptm}
It holds that $P_+(L^2(\Sigma^c,\R^3))=B_+(D_{+,\Sigma})$ and $P_-(L^2(\Sigma^c,\R^3))=B_-(D_{-,\Sigma})$, with 
\begin{align}
	D_{+,\Sigma}&=L^2(\Sigma^c)+\left(K+\frac{1}{2}I\right)V_\Sigma,\label{eqn:dsp}
	\\D_{-,\Sigma}&=L^2(\Sigma^c)+\left(K-\frac{1}{2}I\right)V_\Sigma.\label{eqn:dsm}
\end{align}
Both $D_{+,\Sigma}$ and $D_{-,\Sigma}$ are dense but not closed in $L^2(\SS)$ and, consequently, $B_+(D_{+,\Sigma})$ and $B_-(D_{-,\Sigma})$ are dense but not closed in $H_+(\SS)$ and $H_-(\SS)$, respectively. Furthermore, we define the linear mappings $\tau_{ptm}:D_{+,\Sigma}\to D_{-,\Sigma}$ and $\mtp:D_{-,\Sigma}\to D_{+,\Sigma}/\langle1\rangle$ by
\begin{align}
	\ptm&=\left[P_\Sigma\left(K+\frac{1}{2}I\right)\right]^{-1}P_\Sigma-I\label{eqn:ptm}
	\\\mtp&=-\left[P_\Sigma\left(K-\frac{1}{2}I\right)\right]^{-1}P_\Sigma-I+\langle\cdot,1\rangle_{L^2(\SS)}.\label{eqn:mtp}
\end{align}
Both operators are surjective and unbounded; they are inverse to each other in the sense $\ptm\circ\mtp=I|_{D_{-,\Sigma}}$ and $\mtp\circ\ptm=I|_{D_{+,\Sigma}/\langle1\rangle}$. It holds that for $f\in D_{+,\Sigma}$, the function $g=\tau_{ptm}(f)$ is the unique function in $L^2(\SS)$ such that $B_+f+B_-g$ differs from a locally supported vector field in $L^2(\Sigma^c,\R^3)$ only by a tangential divergence-free vector field from $D(\SS)$. Vice versa, if $\f$ is in $L^2(\Sigma^c,\R^3)$ and $\f_+=B_+f$ with $f\in D_{+,\Sigma}$, then $\f_-=B_-\tau_{ptm}(f)$.
\end{thm}

\begin{rem}\label{rem:scalarg}
Observing that the operators $B_+$ and $B_-$ are bounded and invertible, Theorem \ref{thm:tauptm} implies that, for the construction of vectorial spherical basis functions in the Hardy subspaces of interest, it suffices to find adequate spherical basis functions in the scalar-valued spaces $D_{+,\zeroset}$ and $D_{-,\zeroset}$. Therefore, for the remainder of the paper, in order to simplify notations, we focus on the scalar setup of the spaces $D_{+,\zeroset}$, $D_{-,\zeroset}$, and the operators $\ptm$, $\mtp$. The corresponding results for the vectorial setup are obtained immediately by application of the operators $B_+$ and $B_-$. 
\end{rem}

\subsection{Norm minimizing localized vector fields}\label{sec:minnorm}

Remembering that the original motivation to study localization constraints came from the study of non-uniqueness properties for inverse magnetization problems, we may ask the following question: given a magnetic field in the exterior of the sphere $\SS$, what is the magnetization of minimal norm that is localized within a given source region and that produces the given magnetic field? In terms of the paper at hand, the corresponding question would be: given the $H_+(\SS)$-contribution of some vector field $\f$ in $L^2(\Sigma^c,\R^3)$, what is the vector field $\tilde{\f}$ in $L^2(\Sigma^c,\R^3)$ of minimal norm such that $f_+=\tilde{f}_+$? This question is answered in the next proposition.

\begin{prop}\label{prop:minnorm}
	Let $f$ be in $D_{+,\Sigma}$ and set $\tilde{\f}=B_+f+B_-\ptm(f)+\tilde{\f}_{df}$ with
	\begin{align}\label{eqn:fdf}
		\tilde{\f}_{df}=\left\{\begin{array}{ll}-B_+f-B_-\ptm(f),&\textnormal{ on }\Sigma, \\-\nabla_\SS h,&\textnormal{ on }\Sigma^c,\end{array}\right.
	\end{align}
	and $h\in L^2(\Sigma^c)$ such that $\Delta_\SS h=0$ on $\Sigma^c$ and $\boldsymbol{\nu}\cdot\nabla_\SS h=\boldsymbol{\nu}\cdot(B_+f+B_-\ptm(f))$ on $\partial \Sigma$ (the latter two properties are meant in the distributional sense; and $\boldsymbol{\nu}$ denotes the unit vector that is tangential to $\SS$, normal to the boundary $\partial\Sigma$, and points into the exterior of $\Sigma$, i.e., into $\Sigma^c$). Then it holds that $\tilde{\f}$ is in $L^2(\Sigma^c,\R^3)$, $\tilde{\f}_{df}$ is in $H_{df}(\SS)$, and
	\begin{align}\label{eqn:minnorm}
		\scalemath{0.95}{\|\tilde{\f}\|_{L^2(\SS,\R^3)}=\min\left\{\|\f\|_{L^2(\SS,\R^3)}:\f=B_+f+\g+\mathbf{d} \in L^2(\Sigma^c,\R^3),\,\g\in H_-(\SS),\,\mathbf{d}\in H_{df}(\SS)\right\}.}
	\end{align}
\end{prop}

\begin{proof}
	The statement that $\f_{df}$ of the form \eqref{eqn:fdf} exists and is in $H_{df}(\SS)$ has been shown in \cite[Lemma A.1]{gerhuakeg23}. By construction it is then clear that $\tilde{\f}$ vanishes on $\Sigma$, i.e., $\tilde{\f}$ is in $L^2(\Sigma^c,\R^3)$. It remains to show \eqref{eqn:minnorm}. From the previous section it is known that $\g=B_-\ptm (f)$ is necessary in order to enable $B_+f+\g+\mathbf{d}$ to be in $L^2(\Sigma^c,\R^3)$, so that only the choice of $\d$ needs to be investigated. Due to the orthogonality of $H_+(\SS)$, $H_-(\SS)$, and $H_{df}(\SS)$, we know
	\begin{align}
		\|B_+f+B_-\ptm(f)+\mathbf{d}\|_{L^2(\SS,\R^3)}^2=\|B_+f+B_-\ptm(f)\|_{L^2(\SS,\R^3)}^2+\|\mathbf{d}\|_{L^2(\SS,\R^3)}^2.
	\end{align}
	Furthermore, any $\mathbf{d}$ that guarantees $B_+f+B_-\ptm(f)+\mathbf{d}$ to vanish on $\Sigma$ must coincide with $\tilde{\f}_{df}$ on $\Sigma$. Thus, we obtain $\d=\tilde{\f}_{df}+\tilde{\d}$, for some $\tilde{\d}\in H_{df}(\SS)$ with $\supp(\tilde{\d})\subset\Sigma^c$. We can compute further,
	\begin{align}
		\|\mathbf{d}\|_{L^2(\SS,\R^3)}^2&=\|\tilde{\f}_{df}\|_{L^2(\SS,\R^3)}^2+\|\tilde{\mathbf{d}}\|_{L^2(\SS,\R^3)}^2+2\langle \tilde{\f}_{df},\tilde{\d}\rangle_{L^2(\SS,\R^3)}
		\\&=\|\tilde{\f}_{df}\|_{L^2(\SS,\R^3)}^2+\|\tilde{\mathbf{d}}\|_{L^2(\SS,\R^3)}^2.\nonumber
	\end{align} 
	The last summand in the first row vanishes due to the local support condition on $\tilde{\d}$ and the fact that $\tilde{\f}_{df}$ is expressible as a gradient on $\Sigma^c$. The minimization of the norm of $\d$ therefore leads to $\tilde{\d}=0$, which concludes the proof.
\end{proof}

\section{Spherical basis functions related to Hardy spaces}\label{sec:sbf}

We now consider spherical basis functions contained in and suited for approximation in $D_{+,\zeroset}$ and $D_{-,\zeroset}$, respectively. Remembering the characterizations  \eqref{eqn:dsp} and \eqref{eqn:dsm}, we first deal with approximation in the space $L^2(\zeroset^c)$, then with approximation in $V_\zeroset$, and, subsequently, in Section \ref{sec:approxhardy} we join both considerations to deal with approximation in the actual spaces $D_{+,\zeroset}$ and $D_{-,\zeroset}$. From Remark \ref{rem:scalarg} we know that this suffices for approximation in the corresponding subspaces of $H_+(\SS)$ and $H_-(\SS)$. Again, throughout, we assume that $\Sigma\subset \SS$ is a Lipschitz domain with connected boundary.

\subsection{Approximation in $L^2(\Sigma^c)$}\label{sec:approxl2s}

Simply for notational convenience, we drop the index "$c$" in this subsection and consider approximation in $L^2(\zeroset)$ instead of $L^2(\zeroset^c)$. Obviously, all results equally hold true on $\Sigma^c$ just by taking the complement. We prove the upcoming theorem following step by step the argumentation in the proof of the main result in \cite{towwen13} for the Euclidean setup, solely adapting it to the spherical setup where necessary. The main framework for the spherical setup has already been described in Section \ref{sec:spherbasis}, based on, e.g., \cite{legnarwar06,LeGia2017,legia10} and the stereographic projection.

\begin{thm}\label{thm:l2sigmaapprox}
Let $X_n\subset \SS$ be pointsets with associated mesh widths $h_n=h_{X_n,\Sigma}$ and separation widths $q_n=q_{X_n,\Sigma}$ that satisfy $c_1\gamma h_n\leq h_{n+1}\leq \gamma h_n$ and $q_n\leq h_n\leq c_2 q_n$, for some fixed constants $c_1,\gamma\in(0,1)$ and $c_2\geq 1$. Furthermore, let $\Psi$ be as in  Proposition \ref{prop:spectraofdilatedkernel} with native space $H^{s}(\SS)$, for some $s>1$, and choose scaling factors $\delta_n=\nu h_n$, for some fixed $\nu>1$. Eventually, set $\tilde{X}_n=\{x\in X_n:\mathcal{B}_{\delta_n}(x)\cap\SS\subset \Sigma\}$. Then, for sufficiently large $\nu$ and sufficiently small $\gamma,h_1$, the following holds true: For any $f$ in $H^s(\Sigma)$, there exists a $f_n$ in $\textnormal{span}\{\Psi_{\delta_i,x_i}:x_i\in\tilde{X}_i,\,i=1,\ldots,n\}\subset H^s(\Sigma)$ with
	\begin{align}\label{eqn:errl2wend}
		\|f-f_n\|_{L^2(\SS)}\leq C \eta^{n}\|f\|_{H^{s}(\SS)},
	\end{align}
	where the constants $C>0$ and  $\eta\in(0,1)$ depend on $\nu$, $\gamma$, $c_2$, $h_1$, and $s$.
\end{thm}

Before we proceed to the proof, we mention the following proposition that states some technical properties of Lipschitz domains with a tube around the boundary removed. 

\begin{prop}\cite[Lemma 4.2, Lemma 4.3]{towwen13}\label{prop:seperationofLipschitzdomain}
	Let $\Omega\subset\R^2$ be a bounded Lipschitz domain that satisfies the interior cone condition with parameters $r$ and $\theta$ (reflecting the radius and the opening angle of the cone). Let $c>\max\{2,\frac{\pi}{\theta \sin(\pi/8)}\}$ be fixed. Then there exist constants $C,\delta_0>0$ so that, for any $\delta\in(0,\delta_0)$, there exists a Lipschitz domain $K_{\delta}\subset\Omega$ satisfying
	\begin{enumerate}
		\item $\Omega_{c\delta}\subset K_{\delta}\subset \Omega_\delta$, where $\Omega_\delta=\{x\in\Omega:\inf_{y\in\partial \Omega}\|x-y\|>\delta\}$,
		\item $\mu(\Omega\setminus K_{\delta}) \leq C\,c\delta$, where $\mu$ denotes the Lebesgue measure,
		\item $ K_{\delta}$ satisfies the interior cone condition with parameters $\tilde{\theta} =\min\{\frac{\pi}{5},\frac{\theta}{2}\}$ and $\tilde{r}=\frac{3r\sin\tilde{\theta}}{8(1+1\sin\tilde{\theta})}$.
	\end{enumerate}
\end{prop}

\begin{prfof}{Theorem \ref{thm:l2sigmaapprox}} We choose $f_n$ to be the function obtained by a multiscale interpolation procedure as indicated, e.g., in \cite{legia10}. Namely, set $e_0=f$, iterate by setting, for $i=1,\ldots, n$,
	\begin{align}
		s_i=I_{\tilde{X}_i, \Psi_{\delta_{i}}}(e_{i-1}), \qquad e_i=e_{i-1}-s_i,
	\end{align}
	and eventually choose $f_n=\sum_{i=1}^n s_i=f-e_n$. By $I_{\tilde{X}_i, \Psi_{\delta_{i}}}$ we mean the interpolation operator with respect to the set $\tilde{X}_i$ and the scaled spherical basis function $\Psi_{\delta_i}$, as indicated in Section \ref{sec:spherbasis}. We now follow the steps from \cite{towwen13} in order to estimate $\|e_n\|_{L^2(\SS)}$ in our spherical setup. 
	
	First, we observe the following chain of inequalities:
	\begin{align}
		\|e_{i+1}\|_{\Psi_{\delta_{i+1}}} \leq C\|e_{i+1}\|_{\sobs{s}}\leq C\delta_{i+1}^{-s} \|e_{i+1}\|_{\Psi_{\delta_{i+1}}} \leq C\delta_{i+1}^{-s} \|e_i\|_{\Psi_{\delta_{i+1}}},
	\end{align}
	where the first two inequalities follow form Proposition \ref{prop:spectraofdilatedkernel}, and the last one follows form Proposition \ref{prop:projectioninnativespace} since $e_{i+1}=e_i-I_{\tilde{X}_{i+1},\Psi_{\delta_{i+1}}}(e_i)$. The constant $C>0$ depends on $s$ (throughout the course of this proof, we consider it a generic constant that may change; and only if its dependencies change, we will mention this explicitly). We can split $\|e_i\|_{\Psi_{\delta_{i+1}}}$ into
	\begin{align}\label{eqn:eip1est}
		\|e_i\|_{\Psi_{\delta_{i+1}}}^2 & =  \sum_{n=0}^\infty\sum_{k=-n}^{n}\frac{|(e_{i})^\wedge_{n,k}|^2}{\left(\Psi_{\delta_{i+1}}\right)^\wedge_n}\\ \nonumber
		&\leq c^{-1} \Bigg(\underbrace{\sum_{n\leq \delta_{i+1}^{-1}}\sum_{k=-n}^n|(e_{i})^\wedge_{n,k}|^2 (1+\delta_{i+1}n)^{2s}}_{=I_1} +\underbrace{\sum_{n> \delta_{i+1}^{-1}}\sum_{k=-n}^n|(e_{i})^\wedge_{n,k}|^2 (1+\delta_{i+1}n)^{2s}}_{=I_2}\Bigg),
	\end{align}
	where the constant $c>0$ ist the lower bound in \eqref{eqn:spectradecay} of Proposition \ref{prop:spectraofdilatedkernel} that depends on $s$. 
	
	For the low frequency part $I_1$, we get 
	\begin{align}
		I_1 \leq 2^{2s} \|e_{i}\|_{L^2(\SS)}^2 =2^{2s} \|e_{i}\|_{L^2(\Sigma)}^2 \leq C\,2^{2s} \|\sgp e_{i}\|_{L^2(\sgp(\Sigma))}^2.
	\end{align}
	The second equality follows from the observation that $e_i$ is locally supported in $\Sigma$ and the last inequality from Proposition \ref{prop:propertyofstereographicprojection}. Setting $\Omega=\sgp(\Sigma)$ and letting $K_{\tilde{C}(\delta_i+h_i)}$ be the set from Proposition \ref{prop:seperationofLipschitzdomain} (where $\tilde{C}>0$ denotes the constant from the upper bound in \eqref{eqn:upboundstereopt} of Proposition \ref{prop:propertyofstereographicprojection}), we can now further estimate
	\begin{align}
		I_1  \leq \underbrace{C 2^{2s} \|\sgp e_{i}\|_{L^2(K_{\tilde{C}(\delta_i+h_i)})}^2}_{=I_{1,1}} + \underbrace{C 2^{2s}\|\sgp e_{i}\|_{L^2(\Omega \setminus K_{\tilde{C}(\delta_i+h_i)})}^2}_{=I_{1,2}}
	\end{align}
	Using the notation from Proposition \ref{prop:seperationofLipschitzdomain} $i)$, we get
	\begin{align}
		K_{\tilde{C}(\delta_i+h_i)} &\subset \Omega_{\tilde{C}(\delta_i+h_i)}\subset \bigcup_{x\in\tilde{X}_i}\sgp(\mathcal{B}_{h_i}(x)\cap\SS)\subset \bigcup_{x\in\tilde{X}_i}\mathcal{B}_{\tilde{C}h_i}(\sgp(x)).
	\end{align}
	The above implies $h_{\sgp(\tilde{X}_i),K_{\tilde{C}(\delta_i+h_i)}}\leq \tilde{C}h_i$ for the mesh width. Additionally, we have by construction that $\sgp e_i|_{\sgp(\tilde{X}_i)}=0$. Therefore, we can use \cite[Theorem 2.12]{narwarwen05} to obtain the estimate $\|\sgp e_{i}\|_{L^2(K_{\tilde{C}(\delta_i+h_i)})}\leq \hat{C}(\tilde{C}h_i)^s \|\sgp e_{i}\|_{H^s(K_{\tilde{C}(\delta_i+h_i)})}$, where the constant $\hat{C}>0$ depends on $s$ and $\Omega$ but indendent of $i$, due to the interior cone property from Proposition \ref{prop:seperationofLipschitzdomain} $iii)$ and the discussion in \cite{narwarwen05}. Thus, we can continue to estimate
	\begin{align}\label{eqn:I11}
		I_{1,1} & \leq C h_i^{2s} \|\sgp e_i\|^2_{H^{s}(K_{\tilde{C}(\delta_i+h_i)})}\nonumber 
		\leq C h_i^{2s} \|\sgp e_i\|^2_{H^{s}(\Omega)}\\ 
		& \leq C h_i^{2s} \|e_i\|^2_{H^{s}(\SS)}\leq C \left(\frac{h_i}{\delta_i}\right)^{2s}  \|e_{i}\|_{\Psi_{\delta_{i}}}^2 \\ \nonumber
		&\leq  C \left(\frac{h_i}{\delta_i}\right)^{2s}  \|e_{i-1}\|_{\Psi_{\delta_{j}}}^2= C \nu^{-2s}  \|e_{i-1}\|_{\Psi_{\delta_{j}}}^2,
	\end{align}
	where the third inequality follows from Proposition \ref{prop:propertyofstereographicprojection}, the fourth from Proposition \ref{prop:spectraofdilatedkernel}, and the fifth from Proposition \ref{prop:projectioninnativespace}.	$I_{1,2}$ can be controlled using Propositions \ref{prop:seperationofLipschitzdomain} $ii)$, \ref{prop:propertyofstereographicprojection}, and \ref{Linftyboundedness}:
	\begin{align}\label{eqn:I12}
		I_{1,2} & \leq C  \mu\left(\Omega\setminus K_{\tilde{C}(\delta_{i}+h_i)}\right)\|\sgp e_i\|_{L^{\infty}(\Omega)}^2 \leq C  (\delta_i+h_i)\|e_i\|_{L^{\infty}(\SS)}^2 \nonumber\\ 
		& \leq C(\delta_i+h_i) \left(1+\bar{C}\right)^{2i}\|e_0\|_{L^{\infty}(\SS)}^2=  C\delta_i(1+\nu^{-1}) \left(1+\bar{C}\right)^{2i}\|f\|_{L^{\infty}(\SS)}^2\\ \nonumber
		& \leq C\frac{\delta_1}{\gamma}(1+\nu^{-1}) \left(\sqrt{\gamma}(1+\bar{C})\right)^{2i}\|f\|_{H^s(\SS)}^2=C\frac{\delta_1}{\gamma}(1+\nu^{-1}) \,\eta^{2i}\|f\|_{H^s(\SS)}^2,
	\end{align}
	where we have set $\eta=\sqrt{\gamma}(1+\bar{C})$. For the third inequality one needs to observe that $e_i=e_{i-1}-I_{\tilde{X}_i,\tilde{\Psi}_{\delta_i}}(e_{i-1})$ before applying Proposition \ref{Linftyboundedness} (the constant $\bar{C}>$ depends on $c_2$, $\nu$, and $s$). For the last inequality we have used that $H^s(\SS)$ is compactly embedded in $C^0(\SS)$, for $s>1$. 
	
	For the high frequency contribution $I_2$ we observe that, for $n>\delta_{i+1}^{-1}$,
	\begin{align*}
		(1+\delta_{i+1} n)^{2s}\leq (2 \delta_{i+1} n)^{2s}\leq 2^{2s} \left(\frac{\delta_{i+1}}{\delta_{i}}\right)^{2s}(1+\delta_{i}n)^{2s}.
	\end{align*}	
	Thus, it holds
	\begin{align}\label{eqn:I2}
		I_2  &\leq  2^{2s}\left(\frac{\delta_{i+1}}{\delta_{i}}\right)^{2s} \sum_{n=0}^\infty \sum_{k=-n}^{n}(e_i)^\wedge_{n,k}|^2(1+\delta_i n)^{2s}\nonumber \\ 
		&\leq  C\left(\frac{\delta_{i+1}}{\delta_{i}}\right)^{2s} \sum_{n=0}^\infty \sum_{k=-n}^{n}\frac{|(e_i)^\wedge_{n,k}|^2}{\left(\Psi_{\delta_{i}}\right)^\wedge_n}\\ \nonumber
		&\leq  C\gamma^{2s} \|e_{i}\|_{\Psi_{\delta_{i}}}^2 \leq C\gamma^{2s} \|e_{i-1}\|_{\Psi_{\delta_{i}}}^2 
	\end{align}
	For the second inequality we have used Proposition \ref{prop:spectraofdilatedkernel} and for the last one Proposition \ref{prop:projectioninnativespace}.
	
	Eventually, combining \eqref{eqn:I11}, \eqref{eqn:I12}, \eqref{eqn:I2} with \eqref{eqn:eip1est}, we obtain 
	\begin{align}\label{eqn:concatingI}
		\|e_i\|_{\Psi_{\delta_{i+1}}}^2 \leq C(\nu^{-2s} + \gamma^{2s})  \|e_{i-1}\|_{\Psi_{\delta_{i}}}^2 + C\frac{\delta_1(1+\nu^{-1})}{\gamma} \eta^{2i}\|f\|_{H^{s}(\SS)}^2.
	\end{align}
	For brevity, we set $\alpha^2=C(\nu^{-2s} + \gamma^{2s})$ and $C'=C\frac{\delta_1(1+\nu^{-1})}{\gamma}$. Iterating the estimate with respect to $i$ until $i=0$ is reached, this leads to
	\begin{align}\label{eqn:temp}
		\|e_i\|_{\Psi_{\delta_{i+1}}}^2 & \leq \alpha^{2(i-1)}\|f\|_{\Psi_{\delta_{1}}}^2+	C'\|f\|_{H^{s}(\SS)}^2 \sum_{k=1}^{i}\eta^{2k}\alpha^{2(i-k)}\\
		&\leq C \alpha^{2(i-1)}\|f\|_{H^{s}(\SS)}^2+	C'\frac{\eta^2(\eta^{2i}-\alpha^{2i})}{\eta^{2}-\alpha^{2}}\|f\|_{H^{s}(\SS)}^2, \nonumber
	\end{align}
	where Proposition \ref{prop:spectraofdilatedkernel} is used for the second inequality. Remembering that $\eta=\sqrt{\gamma}(1+\bar{C})$, one can choose $\nu$ sufficiently large and $\gamma$
	sufficiently small such that $\alpha<\eta$ and $\eta<1$ (note that $\bar{C}$ depends on $\nu$, so that one would first choose $\nu$ and subsequently adapt $\gamma$). Possibly increasing the constants $C$ and $C'$ appropriately, one can now cancel the contributions in \eqref{eqn:temp} that contain a factor $\alpha^{2i}$, and using the fact that $\|\cdot\|_{\Psi_{\delta_{i}}}$ uniformly controls $\|\cdot\|_{L^2(\SS)}$ regardless of varying $i$ by \eqref{eqn:spectradecay}, we end up with the desired result
	\begin{align}\label{eqn:errsimple}
		\|f-f_n\|_{L^2(\SS)}=\|e_n\|_{L^2(\SS)}\leq C\eta^{n}\|f\|_{H^{s}(\SS)}
	\end{align}
	The final constant $C$ as well as the previously fixed $\eta\in(0,1)$ depend on $\nu$, $\gamma$, $s$, $c_2$, and $h_1$(since $\delta_1=\nu h_1$). 
\end{prfof}

\subsection{Approximation in $\divfreezeroset$}\label{sec:approxvsigma}

We consider approximation in $V_\zeroset$ based on regularized fundamental solutions as introduced in Section \ref{sec:reggreen}.  First, a general approximation property of the regularized fundamental solution in Sobolev spaces on the sphere is shown. Afterwards, we proceed to locally harmonic functions on Lipschitz domains $\Sigma\subset\SS$.

\begin{prop}\label{lem:approximationbyregularizedgreenfunction}
	Let $L\in \mathbb{N}$ and $X\subset \SS$ be such that condition \eqref{eqn:meshnormcond} is satisfied. Furthermore, let $\kappa\in(0,\frac{1}{2}]$ and $\bar{x}\in\SS$ be fixed. Then, for $\rho\in(0,2)$, the following holds true: For every  $f$ in $\sobs{3+\kappa}$, there exists a $f_{\rho,L}$ in $\textnormal{span}\{1_{\SS},G^{\rho}_{x,\bar{x}}(\DeltaS;\cdot): x\in  X\}$ such that
	\begin{align}
		\|f-f_{\rho,L}\|_{\sobs{1}}\leq C\left(\rho^{\frac{1}{2}}+\rho ^{-\frac{1}{2}}L^{-(1+\kappa)}\right)\|\DeltaS f\|_{\sobs{1+\kappa}},\label{eqn:estsvsigma}
	\end{align}
    with a constant $C>0$ depending  on $\kappa$.
\end{prop}

\begin{proof}
	We can assume $\langle f,1\rangle_{L^2(\SS)}=0$ for the remainder of the proof since constant functions are obviously contained in $\textnormal{span}\{1_{\SS},G^{\rho}_{x,\bar{x}}(\DeltaS;\cdot): x\in  X\}$. For ease of notation, we write $X=\{x_1,\ldots,x_N\}$, and $Q_X$ be a cubature rule with corresponding (not yet specified) weights $w_1,\ldots, w_N$. Setting $F=\DeltaS f$, we get $F\in\sobs{1+\kappa}$ and by Green's third theorem, we have $f=\int_{\SS}G(\DeltaS; y,\cdot)F(y)d\omega(y)$. In partiuclar, $F$ is continuous and allows pointwise evaluations. Finally, observing that $\int_{\SS}G^\rho(\DeltaS; \bar{x},\cdot)F(y)d\omega(y)=0$, we can write
	\begin{align}
		&	\left\|f-\sum_{x_i\in X}F(x_i)w_i\, G^{\rho}_{x_i,\bar{x}}\right\|_{\sobs{1}}\nonumber\\
		&=  \left\|\int_{\SS}G(\DeltaS; y,\cdot)F(y)d\omega(y)-\sum_{x_i\in X}F(x_i)w_i \, G^{\rho}_{x_i,\bar{x}}(\DeltaS;\cdot)\right\|_{\sobs{1}}\nonumber\\
		&\leq \underbrace{\left\|\int_{\SS}G(\DeltaS; y,\cdot)F(y)d\omega(y)-\int_{\SS}G^{\rho}(\DeltaS: y,\cdot)F(y)d\omega(y)\right\|_{\sobs{1}}}_{=I_1}\label{eqn:i1i2}\\
		&\quad+\underbrace{\left\|\int_{\SS}G^{\rho}_{y,\bar{x}}(\DeltaS: \cdot)F(y)d\omega(y)-\sum_{x_i\in  X}F(x_i)w_i\, G^{\rho}_{x_i,\bar{x}}(\DeltaS: \cdot)\right\|_{\sobs{1}}}_{=I_2}.\nonumber
	\end{align}
	The $I_1$ contribution can be estimated using Proposition \ref{prop:regularizedgreenfunction} ii):
	\begin{align}
		I_1^2&\leq \left\|\int_{\SS}G(\DeltaS; y,\cdot)F(y)d\omega(y)-\int_{\SS}G^{\rho}(\DeltaS: y,\cdot)F(y)d\omega(y)\right\|_{L^2(\SS)}^2\nonumber\\
		&\quad +\left\|\gradS\int_{\SS}G(\DeltaS; y,\cdot)F(y)d\omega(y)-\gradS\int_{\SS}G^{\rho}(\DeltaS: y,\cdot)F(y)d\omega(y)\right\|_{L^2(\SS,\R^3)}^2\nonumber\\
		&\leq C^2\|F\|_{C^{0}(\SS)}^2 \left(\rho^2 |\ln(\rho)|^2+\rho\right)\leq C^2\|F\|_{C^{0}(\SS)}^2\rho \leq C^2\|F\|_{H^{1+\kappa}(\SS)}^2\rho,\label{eqn:esti1}
	\end{align}
where the last estimate follows from the compact embedding of $H^s(\SS)$ in $C^{0}(\SS)$ for $s>1$. Throughout the course of the proof, $C>0$ is a generic constant that may change. In order to estimate $I_2$, we first observe that, since condition \eqref{eqn:meshnormcond} is satisfied by assumption, Theorem \ref{thm:cubature2} guarantees the existence of a cubature rule $Q_X$ with polynomial precision of degree $L$ and positive weights $w_1,\ldots, w_N>0$. From now on we fix the weights appearing in our computations to those corresponding to such a particular cubature rule. We can estimate the $\sobs{1}$ norm using the principle of duality 
	\begin{align}
		I_2&=\left\|\int_{\SS}G^{\rho}_{y,\bar{x}}F(y)d\omega(y)-\sum_{x_i\in  X}F(x_i)w_i\, G^{\rho}_{x_i,\bar{x}}(\DeltaS;\cdot)\right\|_{\sobs{1}}\nonumber\\
		&=  \sup_{g\in C^{\infty}(\SS),\atop \|g\|_{\sobs{-1}}\leq 1} \left|\int_{\SS} g(z)\left(\int_{\SS}G^{\rho}_{y,\bar{x}}(\DeltaS; z)F(y)d\omega(y)-\sum_{x_i\in  X}F(x_i)w_i\, G^{\rho}_{x_i,\bar{x}}(\DeltaS;z)\right) d\omega(z) \right|\nonumber
		\\&=  \sup_{g\in C^{\infty}(\SS),\atop\|g\|_{\sobs{-1}}\leq 1} \left|\int_{\SS}\left(\int_{\SS} g(z)G^{\rho}_{y,\bar{x}}(\DeltaS; z)d\omega(z)\right)F(y)d\omega(y)-\sum_{x_i\in  X}F(x_i)w_i \int_{\SS} g(z)G^{\rho}_{x_i,\bar{x}}(\DeltaS; z)d\omega(z) \right|\nonumber\\
		&=  \sup_{g\in C^{\infty}(\SS),\atop\|g\|_{\sobs{-1}}\leq 1} \left|\int_{\SS}H_g^\rho(y)F(y)d\omega(y)-\sum_{x_i\in  X}F(x_i)w_i \,H_g^\rho(x_i) \right|,\label{eqn:esti2}
	\end{align}\normalsize
	where we have used the abbreviation $H_g^\rho=\int_{\SS} g(z)G^{\rho}_{\cdot,\bar{x}}(\DeltaS; z)d\omega(z)$. The last expression in \eqref{eqn:esti2} is precisely the cubature error for the integrand $H_g^\rho \,F$. Thus we can use Theorem \ref{thm:cubature1} to estimate
	\begin{align}
		I_2\leq \frac{C}{L^{1+\kappa}}\|H_g^\rho\,F\|_{\sobs{1+\kappa}}\label{eqn:esti22}
	\end{align}
	The constant $C>0$ appearing here depends on the choice of $\kappa$. Applying Proposition \ref{thm:smoothnessofconvolution} and \ref{thm:smoothnessofmultiplication}, we see that $H_g^\rho \, F$ lies in $\sobs{1+\kappa}$ since we have assumed $\kappa\leq \frac{1}{2}$. In particular, together with H\"older's inequality, we get
	\begin{align}
		\|H^{\rho}_{g}\, F \|_{\sobs{1+\kappa}} & \leq C \|H^{\rho}_{g}\|_{\sobs{1+\kappa}} \|F\|_{\sobs{1+\kappa}}\nonumber\\
		& \leq C \|H^{\rho}_{g}\|_{\sobs{1+\frac{1}{2}}} \|F\|_{\sobs{1+\kappa}}\nonumber\\
		&\leq C \left(\|G^{\rho}(\DeltaS;\bar{x},\cdot)\|_{\sobs{2}}+\|G^{\rho}(\DeltaS;\bar{x},\cdot)\|_{\sobs{1}}\right) \|g\|_{\sobs{-1}} \|F\|_{\sobs{1+\kappa}}\nonumber\\
		&\leq C\rho^{-\tfrac{1}{2}}\|F\|_{\sobs{1+\kappa}}.
	\end{align}
	The last inequality follows from Proposition \ref{prop:regularizedgreenfunction} $i)$ and from the condition in \eqref{eqn:esti2} that $\|g\|_{\sobs{-1}}\leq 1$. Therefore, combining the above with \eqref{eqn:i1i2}, \eqref{eqn:esti1}, and \eqref{eqn:esti2} yields the desired estimate \eqref{eqn:estsvsigma}.
\end{proof}

Based on the above, we can verify the following approximation property in $S\divfreezeroset$, the space of functions that are locally harmonic in a subdomain $\Sigma$. Note that for Proposition \ref{lem:approximationbyregularizedgreenfunction} we would not need to use $G^{\rho}_{x,\bar{x}}(\DeltaS;\cdot)$, with the auxiliary fixed center $\bar{x}\in\SS$, to obtain the desired estimate; solely approximating with $G^{\rho}(\DeltaS;x,\cdot)$ would yield the same estimate. However, when interested in locally harmonic functions, using $G^{\rho}_{x,\bar{x}}(\DeltaS;\cdot)$ is more appropriate because it is locally harmonic itself, namely, $\DeltaSS{y} G^{\rho}_{x,\bar{x}}(\DeltaS;y)=0$ for all $y\notin \mathcal{C}_\rho(x)\cup \mathcal{C}_\rho(\bar{x})$.

\begin{cor}\label{lem:estsvsigma}
Let $L\in \mathbb{N}$ and $X\subset \SS$ be such that the condition \eqref{eqn:meshnormcond} is satisfied, and $\bar{x}\in\SS$ be fixed. Furthermore, let $\kappa\in (0,\frac{1}{2}]$.  Then, for $\rho\in(0,2)$, the following holds true: For every $f$ in $S\divfreezeroset\cap \sobs{3+\kappa}$, there exist a $f_{\rho,L}$ in $\textnormal{span}\{1_\SS,G^{\rho_n}_{x,\bar{x}}: x\in \zeroset^{c}\cap X\}$ such that
	\begin{align}
		\|f-f_{\rho,L}\|_{\sobs{1}}\leq C\left(\rho^{\nicefrac{1}{2}}+\rho ^{-\nicefrac{1}{2}}L^{-(1+\kappa)}\right)\|\DeltaS f\|_{\sobs{1+\kappa}},\label{eqn:estsvsigma2}
	\end{align}
	with a constant $C>0$ depending  on $\kappa$.
\end{cor}

\begin{proof}
We have $F(x)=\DeltaS f(x)=0$ for all $x\in\zeroset$ due to our assumption that $f$ is in $S\divfreezeroset\cap \sobs{3+\kappa}$. Thus, observing \eqref{eqn:i1i2} in the proof of Proposition \ref{lem:approximationbyregularizedgreenfunction}, no $G^{\rho}_{x_i,\bar{x}}$ with $x_i\in \zeroset\cap X$ is required to achieve the estimate from Proposition \ref{lem:approximationbyregularizedgreenfunction}, which already yields the result desired here.
\end{proof}

Before proceeding to the main assertion, we consider the following proposition that helps to guarantee a favorable behaviour of pointsets near the boundary of Lipschitz domains.

\begin{prop}\label{prop:movingpointsinLipschitzdomain}
	Let $\Sigma\subset\SS$, with $\overline{\Sigma}\not=\SS$, be a Lipschitz domain with connected boundary. Then there exist constants $C,r_0>0$ such that, for any $r<r_0$ and $x\in\Sigma$, there exists a $x'\in\Sigma$ satisfying $\mathcal{B}_{r}(x')\cap\SS\subset\Sigma$ and $|x-x'|\leq Cr$. 
\end{prop}

\begin{proof}
Let $R$ and $\theta$ be the radius and angle parameters of the interior cone condition satisfied by $\sgp(\Sigma)\subset \R^2$. Then, for any $h\in(0,(1+\sin \theta)^{-1}R)$ and any $x\in\Sigma$, there exists $y\in\sgp(\Sigma)$ with $\mathcal{B}_{h \sin\theta}(y)\subset\sgp(\Sigma)$ and $\|y-\sgp(x)\|=h$ (see, e.g., \cite[Lemma 3.7]{wendland05}). Now, let $r=\bar{C}^{-1}h\sin\theta$ (where $\bar{C}>0$ is the constant from the upper bound in \eqref{eqn:upboundstereopt} of Proposition \ref{prop:propertyofstereographicprojection}) and take $x'=\sgp^{-1}(y)$. Then it holds that $\mathcal{B}_{r}(x')\cap\SS \subset \sgp^{-1}\mathcal{B}_{h \sin\theta}(y)\subset\Sigma$ and $\|x'-x\|\leq \bar{c}^{-1}\|y-\sgp(x)\| \leq \bar{c}^{-1}\bar{C}(\sin\theta)^{-1} r$  (where $\bar{c}>0$ is the constant from the lower bound in \eqref{eqn:upboundstereopt} of Proposition \ref{prop:propertyofstereographicprojection}). Choosing  $C=\bar{c}^{-1}\bar{C}(\sin\theta)^{-1}$ and $r_0=\bar{C}^{-1}\sin\theta(1+\sin \theta)^{-1}R$ finishes the proof.
\end{proof}

\begin{thm}\label{thm:densevsigma}
	Let $X_n\subset \SS$ and $L_n\in\mathbb{N}$ (with $L_n\to\infty$ monotonically) be such that condition  \eqref{eqn:meshnormcond} is satisfied and, additionally, such that there exists a constant $\bar{c}>0$ with $h_{X_n}>\bar{c}L_n^{-1}$ for all $n\in\mathbb{N}$. Furthermore, we define $\rho_n=L_n^{-2}$ and $\tilde{X}_n=\{x\in X_n:\mathcal{C}_{\rho_n}(x)\subset\Sigma^c\}$, and we assume $\bar{x}\in\Sigma^c$ with $\mathcal{C}_{\rho_1}(\bar{x})\subset\Sigma^c$ and $\kappa\in(0,\frac{1}{2}]$ to be fixed. If $\rho_1$ is sufficiently small, then the following holds true: For every $f$ in $H^{2+\kappa}(\SS)\cap V_\Sigma$, there exists a $f_n$ in $\textnormal{span}\{1_\SS,S^{-1}G^{\rho_n}_{x,\bar{x}}:x\in\tilde{X}_n\}\subset V_\Sigma$ with
	\begin{align}
		\|f-f_n\|_{L^2(\SS)}\leq C h_{X_n}^{\kappa}\|f\|_{H^{2+\kappa}(\SS)},\label{eqn:estreghx}
	\end{align}
where the constant $C>0$ depends on $\kappa$ and $\bar{c}$.
\end{thm}

\begin{proof}
	Due to the definition of $\tilde{X}_n$ and the choice of $\bar{x}$, the properties of $G^{\rho_n}_{x,\bar{x}}$ yield that $S^{-1}G^{\rho_n}_{x,\bar{x}}$ is contained in $V_\zeroset$. 
	
	If $\rho_1$ is sufficiently small, it holds that $\sqrt{2\rho_n}+h_{X_n}<r_0$, where $r_0$ is the constant from Proposition \ref{prop:movingpointsinLipschitzdomain}. Thus, for every $x\in \{y\in\Sigma^c:\mathcal{C}_{\rho_n}(y)\not\subset\Sigma^c\}$, there exists a $x'\in \Sigma^c$ with $\mathcal{B}_{\sqrt{2\rho_n}+h_{X_n}}(x')\cap\SS\subset\Sigma^c$ and $|x-x'|\leq C'(\sqrt{2\rho_n}+h_{X_n})$ (where $C'>0$ is the constant from Proposition \ref{prop:movingpointsinLipschitzdomain}). For this particular $x'$, by construction, there must exist a $\tilde{x}\in\tilde{X}_n$ with $|x-\tilde{x}|<h_{X_n}$. In consequence, this means that, for every $x\in\Sigma^c$, there must exist some $\tilde{x}\in\tilde{X}_n$ with $|x-\tilde{x}|\leq 2h_{X_n}+C'(\sqrt{2\rho_n}+h_{X_n})\leq (2+C'+\sqrt{2}C'\bar{c}^{-1})h_{X_n}$. This leads us to 
	\begin{align}
		h_{\tilde{X}_n,\Sigma^c}\leq \tilde{C} h_{X_n},
	\end{align}
	where $\tilde{C}>0$ depends on $\bar{c}$. Applying a similar argument to $\Sigma$ instead of $\Sigma^c$, we obtain for the overall set $X_n'=\{x\in X_n: x\in \Sigma \textnormal{ or }x\in\tilde{X}_n\}$ that $h_{X'_n}\leq \tilde{C} h_{X_n}$.
	
	From Corollary \ref{lem:estsvsigma} and the bounded invertibility of $S:L^2(\SS)\to \sobs{1}$, we now get the existence of a $f_n$ in $\textnormal{span}\{1_\SS,S^{-1}G^{\rho_n}_{x,\bar{x}}:x\in\tilde{X}_n\}\subset V_\Sigma$ such that 
	\begin{align}
		\|f-f_n\|_{L^2(\SS)}&\leq C\|Sf-Sf_n\|_{\sobs{1}}\nonumber
		\\&\leq C\left(\rho_n^{\frac{1}{2}}+\rho_n ^{-\frac{1}{2}}L_n^{-(1+\kappa)}\right)\|\DeltaS Sf\|_{\sobs{1+\kappa}}\nonumber
		\\&=C\left(\rho_n^{\frac{1}{2}}+\rho_n ^{-\frac{1}{2}} \rho_n^{\frac{1+\kappa}{2}}\right)\|\DeltaS Sf\|_{\sobs{1+\kappa}}\nonumber
		\\&\leq C\rho_n^{\frac{\kappa}{2}}\|f\|_{\sobs{2+\kappa}},\label{eqn:estsvsigma3}
	\end{align}
	where $C>0$ is a generic constant, which may change in every step of estimation but which eventually depends on $\kappa$ and $\bar{c}$. The desired assertion \eqref{eqn:estreghx} follows from the observation that $\rho_n=L_n^{-2}\leq\bar{c}^{-2}h_{X_n}^2$. 
\end{proof}

\subsection{Approximation in the subspaces $D_{\pm,\Sigma}$}\label{sec:approxhardy}

Due to the structure of $D_{+,\Sigma}$ and $D_{-,\Sigma}$, the considerations in the two previous Sections \ref{sec:approxl2s} and \ref{sec:approxvsigma} directly provide the following theorem.

\begin{thm}\label{thm:densedsigma+}
Let $X'_n\subset\SS$ and $\delta_n>0$ satisfy the conditions of Theorem \ref{thm:l2sigmaapprox} (with $\Sigma$ substituted by $\Sigma^c$). Furthermore, let $X_n\subset\SS$ and $\rho_n>0$ satisfy the conditions of Theorem \ref{thm:densevsigma}, and $\bar{x}\in\Sigma^c$ with $\mathcal{C}_{\rho_1}(\bar{x})\subset\Sigma^c$ be fixed. We define the sets $\tilde{X}'_n=\{x'\in X'_n:\mathcal{B}_{\delta_n}(x')\cap\SS\subset\Sigma^c\}$ and $\tilde{X}_n=\{x\in X_n:\mathcal{C}_{\rho_n}(x)\subset\Sigma^c\}$. Additionally, let $\Psi$ be as in Definition \ref{def:dilatioylkernels} with native space $H^{s}(\SS)$, for $s=2+\kappa$ with a fixed $\kappa\in(0,\frac{1}{2}]$. Then the following holds true: For every $f$ in $D_{+,\Sigma}/\langle1\rangle$, there exists a $f_n$ in $\textnormal{span}\big\{(K+\frac{1}{2}I)S^{-1}G^{\rho_n}_{x,\bar{x}},\,\Psi_{\delta_i,x_i'}:\,x\in \tilde{X}_n,\,x_i'\in\tilde{X}'_i,\,i=1,\ldots,n\big\}\subset D_{+,\Sigma}$ with
\begin{align}
	\|f-f_n\|_{L^2(\SS)}\leq C \left(\eta^{n}+h_{X_n}^{\kappa}\right)\left(\|f\|_{\sobs{s}}+\|\ptm (f)\|_{\sobs{s}}\right).\label{eqn:dsigmaest}
\end{align}
Analogously, for every $g$ in $D_{-,\Sigma}$, there exists a $g_n$ in $\textnormal{span}\big\{(K-\frac{1}{2}I)S^{-1}G^{\rho_n}_{x,\bar{x}},\,\Psi_{\delta_i,x_i'}:\,x\in \tilde{X}_n,\,x_i'\in\tilde{X}'_i,\,i=1,\ldots,n\big\}\subset D_{-,\Sigma}$ with
\begin{align}
	\|g-g_n\|_{L^2(\SS)}\leq C \left(\eta^{n}+h_{X_n}^{\kappa}\right)\left(\|\mtp (g)\|_{\sobs{s}}+\|g\|_{\sobs{s}}\right).\label{eqn:dsigmaest2}
\end{align}
The constants $C>0$ and  $\eta\in(0,1)$ depend on the parameters listed in Theorems \ref{thm:l2sigmaapprox} and \ref{thm:densevsigma}.
\end{thm}

\begin{proof}
According to Theorem \ref{thm:tauptm} we can write $f=h+\left(K+\frac{1}{2}I\right)k$ with $h\in L^2(\Sigma^c)$ and $k\in V_\Sigma$. Applying $\tau_{ptm}$ yields $\tau_{ptm}(f)=-h+k-(K+\frac{1}{2}I)k=k-f$. Thus, we obtain
\begin{align}
k=f+\ptm (f) && h=f-\left(K+\frac{1}{2}I\right)k=-\left(K-\frac{1}{2}I\right)f-\left(K+\frac{1}{2}I\right)\ptm (f)
\end{align}
Since $K\pm\frac{1}{2}I:H^s(\SS)\to H^s(\SS)$ is bounded, we have $\|k\|_{\sobs{s}}\leq C \left(\|f\|_{\sobs{s}}+\|\ptm f\|_{\sobs{s}}\right)$ and $\|h\|_{\sobs{s}}\leq C\left(\|f\|_{\sobs{s}}+\|\ptm f\|_{\sobs{s}}\right)$, for some constant $C>0$. Now, let $h_n\in \textnormal{span}\bigcup_{i=1}^n\{\Psi_{\delta_i,x'}:x'\in\tilde{X}'_i\}$ be an approximation of $h$ according to Theorem \ref{thm:l2sigmaapprox} and $k_n\in\textnormal{span}\{S^{-1}G^{\rho_n}_{x,\bar{x}}:x\in \tilde{X}_n\}$ an approximation to $k$ according to Theorem \ref{thm:densevsigma}, then $f_n=h_n+(K+\frac{1}{2}I)k_n$ is in $D_{+,\Sigma}$ and satisfies the estimate \eqref{eqn:dsigmaest}. The estimate  \eqref{eqn:dsigmaest2} follows along the same lines.
\end{proof}

\begin{rem}\label{rem:ptmrbf}
	It can easily be seen from \eqref{eqn:ptm} that 
	\begin{align}
		\tau_{ptm}\left((K+\tfrac{1}{2}I)S^{-1}G^{\rho_n}_{x,\bar{x}}\right)=-(K-\tfrac{1}{2}I)S^{-1}G^{\rho_n}_{x,\bar{x}},\quad \tau_{ptm}\left(\Psi_{\delta_n,x'}\right)=-\Psi_{\delta_n,x'},\nonumber
	\end{align}
	under the assumptions of Theorem \ref{thm:densedsigma+}. Therefore, for any $f$ in the finite-dimensional subspace $\textnormal{span}\{(K+\frac{1}{2}I)S^{-1}G^{\rho_n}_{x,\bar{x}},\,\Psi_{\delta_i,x_i'}:\,x\in \tilde{X}_n,\,x_i'\in\tilde{X}'_i,\,i=1,\ldots,n\}\subset D_{+,\Sigma}$, one directly obtains the corresponding $\ptm(f)$ in $D_{-,\Sigma}$ without much additional computational effort. Yet, one should keep in mind the unboundedness of the operator $\ptm$ when mapping between the infinite-dimensional spaces $D_{+,\Sigma}$ and $D_{-,\Sigma}$.

\end{rem}

\subsection{Norm minimization and vectorial approximation}\label{sec:minnormsbf}

In Section \ref{sec:minnorm} we have briefly discussed norm minimizing vector fields with localization constraints. A corresponding statement with respect to the spherical basis functions considered above is presented next. For a given function $f$ in $D_{+,\Sigma}$ that is expanded in terms of those spherical basis functions, the norm minimizing vector field $\tilde{\f}$ that is supported in $\Sigma^c$ and that satisfies $\tilde{f}_+=f$ can be expressed fairly explicitly.

\begin{defi}\label{def:greenfunc}
	Let $\rho\in(0,2)$ and $x,\bar{x}\in\SS$ be fixed. Then we denote by $N^{\rho}_{x,\bar{x}}:\Sigma^c\to\R$ the solution in $L^2(\Sigma^c)/\langle1\rangle$ to the following Neumann boundary value problem: $\Delta_\SS N^{\rho}_{x,\bar{x}}=0$ in $\Sigma^c$ and $\boldsymbol{\nu}\cdot\nabla_\SS N^{\rho}_{x,\bar{x}}=\boldsymbol{\nu}\cdot\nabla_\SS G^{\rho}_{x,\bar{x}}$ on $\partial\Sigma$. 
\end{defi}

\begin{rem}
If $x,\bar{x}\in \Sigma$, then clearly $N^{\rho}_{x,\bar{x}}=G^{\rho}_{x,\bar{x}}$. Otherwise, if $x,\bar{x}\in \Sigma^c$, which is the relevant case for our setup, an explicit expression for $N^{\rho}_{x,\bar{x}}$ is more tedious to obtain. However, for the special case of $\Sigma^c$ being a spherical cap, which is sufficient for many considerations, an explicit expression of $N^{\rho}_{x,\bar{x}}$ is known: namely, $N^{\rho}_{x,\bar{x}}(y)=\Phi^{(N)}(x,y)-\Phi^{(N)}(\bar{x},y)$, with $\Phi^{(N)}$ as derived in \cite[Sec. 6.4.5]{freedengerhards12} for the construction of a Neumann Green function on spherical caps.
\end{rem}

\begin{prop}\label{prop:minnormsbf}
	Let the general setup be as in Theorem \ref{thm:densedsigma+}. Furthermore, let $f$ be in $\textnormal{span}\big\{(K+\frac{1}{2}I)S^{-1}G^{\rho_n}_{x,\bar{x}},\,\Psi_{\delta_i,x_i'}:\,x\in \tilde{X}_n,\,x_i'\in\tilde{X}'_i,\,i=1,\ldots,n\big\}\subset D_{+,\Sigma}$, i.e.,
	\begin{align}
		f=\sum_{\ell=1}^M c_\ell \left(K+\frac{1}{2}I\right)S^{-1}G^{\rho_n}_{x_\ell,\bar{x}}+\sum_{i=1}^n\sum_{k=1}^{K_i} c_{i,k} \Psi_{\delta_i,x_{i,k}'},
	\end{align} 
	with $c_\ell,c_{i,k}\in\R$ and $M,K_1,\ldots,K_n\in\mathbb{N}$. For $\tilde{\f}=B_+f+B_-\tilde{f}_-+\tilde{\f}_{df}$ with
	\begin{align}\label{eqn:repsbftildeg}
		\tilde{f}_-=-\sum_{\ell=1}^M c_\ell \left(K-\frac{1}{2}I\right)S^{-1}G^{\rho_n}_{x_\ell,\bar{x}}-\sum_{i=1}^n\sum_{k=1}^{K_i} c_{i,k} \Psi_{\delta_i,x_{j,k}}
	\end{align} 
	and
	\begin{align}\label{eqn:fdf2}
		\tilde{\f}_{df}=\left\{\begin{array}{ll}-\sum_{\ell=1}^M c_\ell \nabla_\SS G^{\rho_n}_{x_\ell,\bar{x}},&\textnormal{ on }\Sigma, \\-\sum_{\ell=1}^M c_\ell \nabla_\SS N^{\rho_n}_{x_\ell,\bar{x}},&\textnormal{ on }\Sigma^c,\end{array}\right.
	\end{align}
	the following holds true: $\tilde{\f}$ is in $L^2(\Sigma^c,\R^3)$, $\tilde{\f}_{df}$ is in $H_{df}(\SS)$, and
	\begin{align}\label{eqn:minnorm2}
		\scalemath{0.95}{\|\tilde{\f}\|_{L^2(\SS,\R^3)}=\min\left\{\|\f\|_{L^2(\SS,\R^3)}:\f=B_+f+\g+\mathbf{d} \in L^2(\Sigma^c,\R^3),\,\g\in H_-(\SS),\,\mathbf{d}\in H_{df}(\SS)\right\}.}
	\end{align}
\end{prop}

\begin{proof}
	We already know that $\tilde{f}_-=\ptm(f)$ is necessary in order to enable $\tilde{\f}$ to be in $L^2(\Sigma^c,\R^3)$. The representation \eqref{eqn:repsbftildeg} then directly follows from Remark \ref{rem:ptmrbf}. Furthermore, we can compute
	\begin{align}
		\scalemath{0.95}{B_+f+B_-\tilde{f}_-=}&\scalemath{0.95}{\sum_{\ell=1}^M c_\ell \boldsymbol{\eta}\left(K-\frac{1}{2}I\right)\left(K+\frac{1}{2}I\right)S^{-1}G^{\rho_n}_{x_\ell,\bar{x}}+\sum_{i=1}^n\sum_{k=1}^{K_i} c_{i,k} \boldsymbol{\eta}\left(K-\frac{1}{2}I\right)\Psi_{\delta_i,x_{i,k}'}}\nonumber
		\\&\scalemath{0.95}{-\sum_{\ell=1}^M c_\ell \boldsymbol{\eta}\left(K+\frac{1}{2}I\right)\left(K-\frac{1}{2}I\right)S^{-1}G^{\rho_n}_{x_\ell,\bar{x}}-\sum_{i=1}^n\sum_{k=1}^{K_i} c_{i,k} \boldsymbol{\eta}\left(K+\frac{1}{2}I\right)\Psi_{\delta_i,x_{i,k}'}}\nonumber
		\\&+\scalemath{0.95}{\sum_{\ell=1}^M c_\ell \nabla_\SS\left(K+\frac{1}{2}I\right)G^{\rho_n}_{x_\ell,\bar{x}}+\sum_{i=1}^n\sum_{k=1}^{K_i} c_{i,k} \nabla_\SS S\Psi_{\delta_i,x_{i,k}'}}\label{eqn:b++b-}
		\\&\scalemath{0.95}{-\sum_{\ell=1}^M c_\ell \nabla_\SS\left(K-\frac{1}{2}I\right)G^{\rho_n}_{x_\ell,\bar{x}}-\sum_{i=1}^n\sum_{k=1}^{K_i} c_{i,k} \nabla_\SS S\Psi_{\delta_i,x_{i,k}'}}\nonumber
		\\\scalemath{0.95}{=}&\scalemath{0.95}{\sum_{\ell=1}^M c_\ell \nabla_\SS G^{\rho_n}_{x_\ell,\bar{x}}+\sum_{i=1}^n\sum_{k=1}^{K_i} c_{i,k} \boldsymbol{\eta}\,\Psi_{\delta_i,x_{i,k}'}.}\nonumber
	\end{align}
	Due to the properties of the point sets $\tilde{X}_i'$ in Theorem \ref{thm:densedsigma+} and the assumption $x_{i,k}'\in \tilde{X}_i'$, the functions $\Psi_{\delta_i,x_{i,k}'}$ vanish on $\Sigma$. Therefore, the expression above implies $B_+f+B_-\tilde{f}_-=\sum_{\ell=1}^M c_\ell \nabla_\SS G^{\rho_n}_{x_\ell,\bar{x}}$ on $\Sigma$. Furthermore, on the boundary $\partial\Sigma$, it holds that
	\begin{align}
		\boldsymbol{\nu}\cdot\left(B_+f+B_-\tilde{f}_-\right)=\sum_{\ell=1}^M c_\ell \boldsymbol{\nu}\cdot\nabla_\SS G^{\rho_n}_{x_\ell,\bar{x}},
	\end{align}
	since $\boldsymbol{\nu}$ is orthogonal to $\boldsymbol{\eta}$ by definition. The choice $h=\sum_{\ell=1}^M c_\ell N^{\rho_n}_{x_\ell,\bar{x}}$ then satisfies $\DeltaS h=0$ on $\Sigma^c$ and $\boldsymbol{\nu}\cdot \nabla_\SS h= \boldsymbol{\nu}\cdot\left(B_+f+B_-\tilde{f}_-\right)$ on $\partial\Sigma$. The assertion of Proposition \ref{prop:minnormsbf} is now a direct consequence of Proposition \ref{prop:minnorm}.
\end{proof}

Furthermore, the above computations also guide us to an error estimate for the joint approximation in $D_{+,\Sigma}$ and $D_{-,\Sigma}$. In fact, this yields a slightly different proof of Theorem \ref{thm:densedsigma+} and it provides a slightly more general statement, since it guarantees the existence of a sequence $f_n$ such that both $f_n$ and $\ptm(f_n)$ satisfy the desired error estimate.

\begin{thm}\label{thm:approxall}
	Let the general setup be as in Theorem \ref{thm:densedsigma+} and let $\f$ be in $L^2(\Sigma^c,\R^3)$. Then there exists a $\f_n$ in $\textnormal{span}\big\{\nabla_\SS G^{\rho_n}_{x,\bar{x}},\,\boldsymbol{\eta}\,\Psi_{\delta_i,x_i'}:\,x\in \tilde{X}_n,\,x_i'\in\tilde{X}'_i,\,i=1,\ldots,n\big\}$, i.e.,  $\f_n=\sum_{\ell=1}^M c_\ell \nabla_\SS G^{\rho_n}_{x_\ell,\bar{x}}+\sum_{i=1}^n\sum_{k=1}^{K_i} c_{i,k} \boldsymbol{\eta}\,\Psi_{\delta_i,x_{i,k}'}$ for some $c_\ell,c_{i,k}\in\R$ and $M,K_1,\ldots,K_n\in\mathbb{N}$, such that
	\begin{align}\label{eqn:vest1}
		\|(\f_++\f_-)-\f_n\|_{L^2(\SS,\R^3)}\leq  C \left(\eta^{n}+h_{X_n}^{\kappa}\right)\left(\|f_+\|_{\sobs{s}}+\|f_-\|_{\sobs{s}}\right),
	\end{align}
where $\f=\f_++\f_-+\f_{df}=B_+f_++B_-f_-+\f_{df}$ reflects the Hardy-Hodge decomposition of $\f$. The constants $C>0$ and  $\eta\in(0,1)$ depend on the parameters listed in Theorems \ref{thm:l2sigmaapprox} and \ref{thm:densevsigma}. 
Furthermore, the functions 
\begin{align}
	f_+^n&=\sum_{\ell=1}^M c_\ell \left(K+\frac{1}{2}I\right)S^{-1}G^{\rho_n}_{x_\ell,\bar{x}}+\sum_{i=1}^n\sum_{k=1}^{K_i} c_{i,k} \Psi_{\delta_i,x_{i,k}'},\label{eqn:vest2}
	\\f_-^n&=\ptm(f_+^n)=-\sum_{\ell=1}^M c_\ell \left(K-\frac{1}{2}I\right)S^{-1}G^{\rho_n}_{x_\ell,\bar{x}}-\sum_{i=1}^n\sum_{k=1}^{K_i} c_{i,k} \Psi_{\delta_i,x_{i,k}'},\label{eqn:vest3}
	\end{align}
with $c_\ell,c_{i,k}\in\R$ and $M,K_1,\ldots,K_n\in\mathbb{N}$ as above, lie in $D_{+,\Sigma}$ and $D_{-,\Sigma}$, respectively, and satisfy
	\begin{align}
		\|f_+-f^n_+\|_{L^2(\SS)}&\leq  {C} \left(\eta^{n}+h_{X_n}^{\kappa}\right)\left(\|f_+\|_{\sobs{s}}+\|f_-\|_{\sobs{s}}\right),\label{eqn:vest4}
		\\\|f_--f^n_-\|_{L^2(\SS)}&\leq  {C}  \left(\eta^{n}+h_{X_n}^{\kappa}\right)\left(\|f_+\|_{\sobs{s}}+\|f_-\|_{\sobs{s}}\right).\label{eqn:vest5}
	\end{align}
\end{thm}

\begin{proof}
	It is well-known known that any $\f$ in $L^2(\SS,\R^3)$ can be decomposed into the Helmholtz decomposition $\f=\boldsymbol{\eta} f_{\boldsymbol{\eta}}+\nabla_\SS S f_{cf}+\f_{df}$, with $f_{\boldsymbol{\eta}},f_{cf}$ in $L^2(\SS)$ and $\f_{df}$ in $H_{df}(\SS)$ (see, e.g., \cite{freedengerhards12,freedenschreiner09}). Since the statement of the theorem assumes that $\f$ is in $L^2(\Sigma^c,\R^3)$, we get that $f_{\boldsymbol{\eta}}$ must be in $L^2(\Sigma^c)$ and $f_{cf}$ must be in $V_\Sigma$. Theorems \ref{thm:l2sigmaapprox} and \ref{thm:densevsigma} then guarantee the existence of an $\f_n^1$ in $\textnormal{span}\left\{\boldsymbol{\eta}\,\Psi_{\delta_i,x_i'}:\,x_i'\in\tilde{X}'_i,\,i=1,\ldots,n\right\}$ and an $\f_n^2$ in $\textnormal{span}\left\{\nabla_\SS G^{\rho_n}_{x,\bar{x}}:\,x\in \tilde{X}_n\right\}$ such that 
	\begin{align}
		\|\boldsymbol{\eta} f_{\boldsymbol{\eta}}-\f_n^1\|_{L^2(\SS,\R^3)}&\leq  C \eta^{n}\|f_{\boldsymbol{\eta}}\|_{\sobs{s}},
		\\\|\nabla_\SS S f_{cf}-\f_n^2\|_{L^2(\SS,\R^3)}&\leq  C h_{X_n}^{\kappa}\|f_{cf}\|_{\sobs{s}}.
	\end{align}
	Combining the two estimates above and observing that $\f_++\f_-=\boldsymbol{\eta} f_{\boldsymbol{\eta}}+\nabla_\SS S f_{cf}$ as well as that $B_+$ and $B_-$ are bounded and invertible, we directly obtain \eqref{eqn:vest1}, with a possibly modified constant $C>0$. The desired estimates \eqref{eqn:vest2} and \eqref{eqn:vest3} follow from \eqref{eqn:vest1} when reading the equations in \eqref{eqn:b++b-} backwards and when observing the mutual orthogonality of the Hardy spaces $H_+(\SS)$ and $H_-(\SS)$. 
\end{proof}

\section{Numerical example}\label{sec:numerics}
\begin{figure}
	\begin{center}
	\includegraphics[scale=0.3]{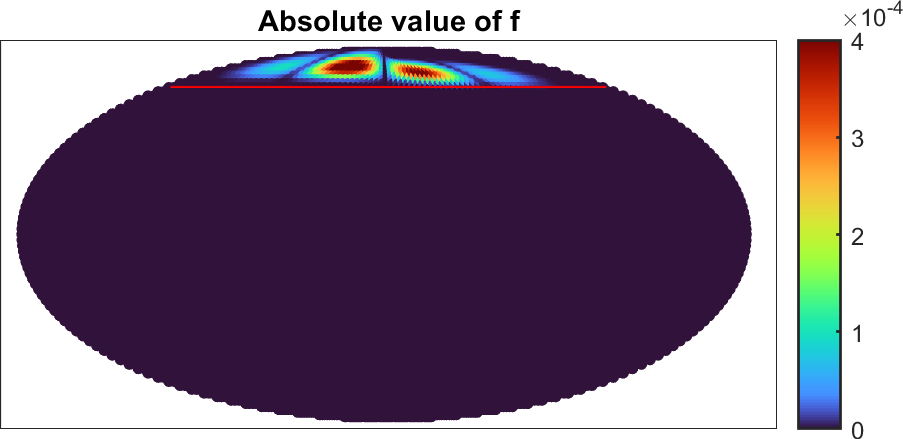}	\includegraphics[scale=0.3]{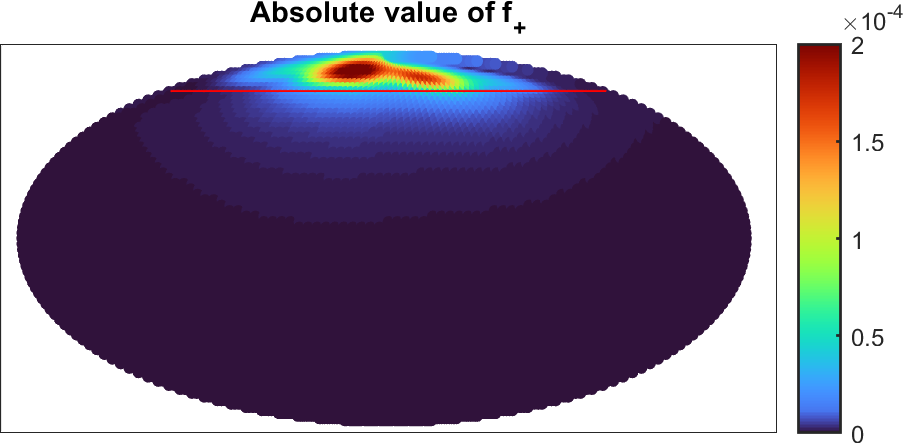}
	\end{center}
	\caption{Absolute values of the underlying vector field $\f$ from \eqref{eqn:f} (\textit{left}) and its $H_+(\SS)$-contribution $\f_+$ (\textit{right}). The red line indicates the boundary of the spherical cap $\mathcal{C}_{0.1}(\mathbf{e}_3)$ that denotes the support of $\f$.}\label{fig:f}
\end{figure}
We provide a brief numerical example that illustrates the suitability of the constructed spherical basis functions for approximation in $D_{+,\Sigma}$ and that reproduces the derived convergence rates. For that purpose, we choose the test vector field $\f(x)=Q(x)(3(x\cdot\d)x-\d)$ with $\d=(0,0.6,0.8)^T$ and
\begin{align}
 Q(x)=Q(t,\varphi)=\left\{\begin{array}{ll}(t-a)^3(t-1)^2(\varphi-2\pi)^3 \sin(2\varphi)\varphi^3,\quad&\textnormal{if }t>a,\\0,&\textnormal{else},\end{array}\right.\label{eqn:f}
\end{align}
noticing that $x=(\sqrt{1-t^2}\cos(\varphi),\sqrt{1-t^2}\sin(\varphi),t)^T\in\SS$ for $t\in[-1,1]$ and $\varphi\in[0,2\pi)$. Clearly, $\f$ is supported in the spherical cap $\mathcal{C}_{1-a}(\mathbf{e}_3)$ with polar radius $1-a\in(0,2)$ and center $\mathbf{e}_3=(0,0,1)^T$. In the upcoming example we choose $a=0.9$, i.e., $\f$ is supported in  $\mathcal{C}_{0.1}(\mathbf{e}_3)$ (the function $\f$ as well as its $H_+(\SS)$-contribution $\f_+=B_+f$ are illustrated in Figure \ref{fig:f}). As a consequence, the corresponding $f$ lies in $D_{+,\Sigma}$, for $\Sigma^c=\mathcal{C}_{0.1}(\mathbf{e}_3)$ but also for $\Sigma^c=\mathcal{C}_{0.2}(\mathbf{e}_3),\,\mathcal{C}_{1}(\mathbf{e}_3)$. In the following we want to approximate this $f$ by solving
\begin{align}
f_n=\textnormal{argmin}\left\{\|f-g\|_{L^2(\SS)}^2+\lambda^2 \|g\|_{H^s(\SS)}^2:g\in \mathcal{D}_n\right\},\label{eqn:optprob}
\end{align} 
with $s=9/4$, a not yet specified regularization parameter $\lambda$, and $\mathcal{D}_n\subset D_{+,\Sigma}$ the finite dimensional subspace
\begin{align}
	\mathcal{D}_n=\textnormal{span}\left\{\left(K+\frac{1}{2}I\right)S^{-1}G^{\rho_n}_{x,\bar{x}},\,\Psi_{\delta_i,x_i'}:\,x\in \tilde{X}_n,\,x_i'\in\tilde{X}'_i,\,i=1,\ldots,n\right\}.\label{eqn:choicedn}
\end{align}
The pointsets $\tilde{X}_n$, $\tilde{X}'_i$ and the parameters $\rho_n$, $\delta_i$ are chosen to satisfy the conditions from Theorem \ref{thm:densedsigma+} (cf. Figure \ref{fig:points} for more details on $\mathcal{D}_n$ and the used pointsets). We run tests for levels $n=1,2, \ldots,6$, and for the three choices of $\Sigma$ indicated above, i.e., for $\Sigma^c=\mathcal{C}_{0.1}(\mathbf{e}_3),\mathcal{C}_{0.2}(\mathbf{e}_3),\,\mathcal{C}_{1}(\mathbf{e}_3)$, while we keep the underlying $f$ the same.  Note that  $\tilde{X}_n$, $\tilde{X}'_i$ depend on $\Sigma$ according to their definition in Theorem \ref{thm:densedsigma+}. Finally, for the evaluation of \eqref{eqn:optprob}, we assume to know $f$ only via its spherical harmonic coefficients up to a maximal degree $100$ (an assumption that is reasonable for typical geomagnetic applications, and which also reflects some noise in the data by neglecting higher degree information).

Figure \ref{fig:errors} indicates the relative errors $\|f_n-f\|_{L^2(\SS)}/(\|f\|_{H^s(\SS)}+\|\tau_{ptm}(f)\|_{H^s(\SS)})$ obtained for the solutions $f_n$ of \eqref{eqn:optprob} within the setup described above (only the outcome for the best choice of $\lambda$, among a tested range on regularization parameters, is plotted). One can see that the evolution of the error with respect to $n$ generally follows the predicted rate from Theorem \ref{thm:densedsigma+}. Additionally one can see that the error is smaller for larger $\Sigma^c$, which is probably due to the increased number of available basis functions contained in $\mathcal{D}_n$. The observation in the left plot of Figure \ref{fig:errors} that the error reaches a plateau at $n=5$ is likely due to the restriction of information on $f$ to spherical harmonic degree 100. If we increase the information on $f$ up to degree 200, the error decreases further as $n$ grows, although it still does not reach the predicted error bound (cf. right plot in Figure \ref{fig:errors}). 
\begin{figure}
	\begin{center}
		\includegraphics[scale=0.22]{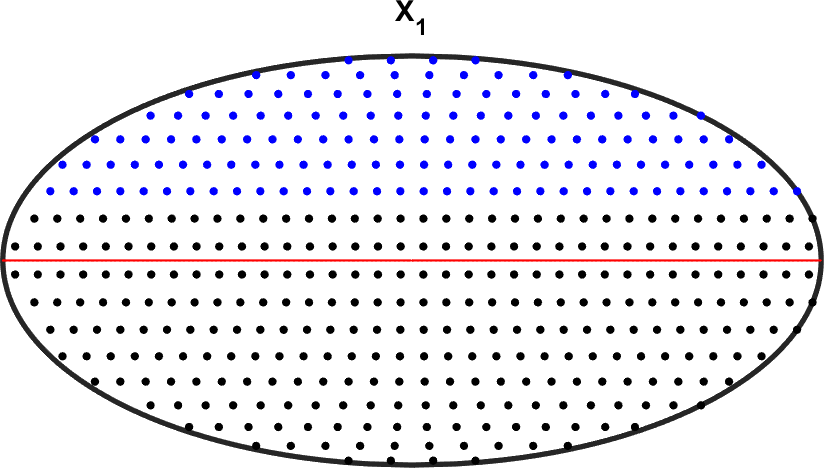}\quad\includegraphics[scale=0.22]{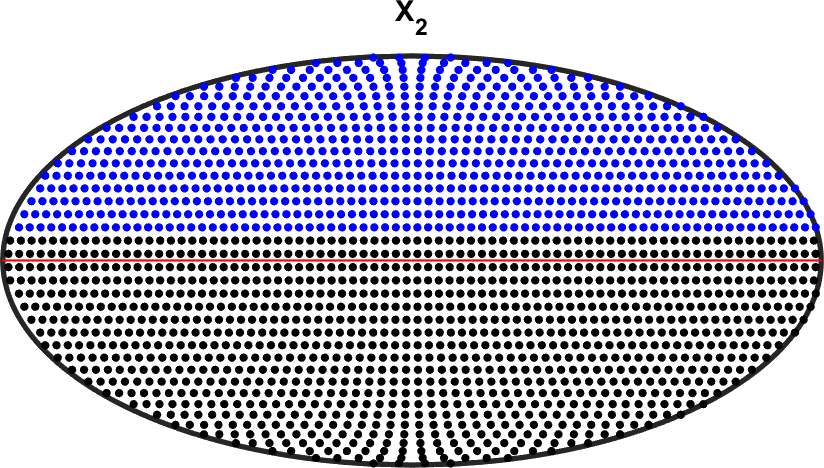}\quad\includegraphics[scale=0.22]{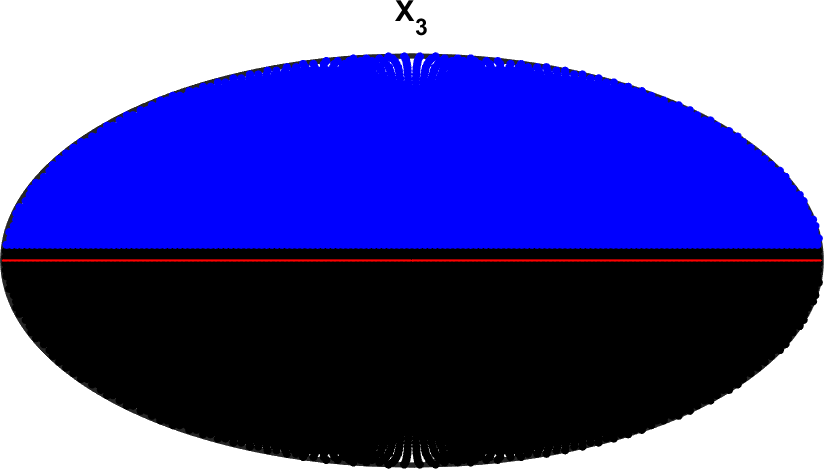}
	\end{center}
	\begin{flushleft}
		\footnotesize
		\begin{tabular}{c | c c c c| c c c c| c c c c|} 
			&&$n=1$&& &&$n=2$&& &&$n=3$&&
			\\&$h_n$&$\#\mathcal{D}_n$&$\delta_n$&$\rho_n$&$h_n$&$\#\mathcal{D}_n$&$\delta_n$&$\rho_n$&$h_n$&$\#\mathcal{D}_n$&$\delta_n$&$\rho_n$
			\\\hline
			$\Sigma_1$& 0.174 & 280 & 0.385 & 0.105 & 0.083 & 1,692 & 0.187 & 0.025 & 0.039 & 8,466 & 0.009 & 0.006
			\\$\Sigma_2$& 0.174 & 8 & 0.385 & 0.105 & 0.083 & 232 & 0.187 & 0.025 & 0.039 & 1,334 & 0.009 & 0.006
			\\$\Sigma_3$& 0.174 & 0 & 0.385 & 0.105 & 0.083 & 60 & 0.187 & 0.025 & 0.039 & 544 & 0.009 & 0.006
			\\\hline\hline
			&&$n=4$&& &&$n=5$&& &&$n=6$&&
			\\&$h_n$&$\#\mathcal{D}_n$&$\delta_n$&$\rho_n$&$h_n$&$\#\mathcal{D}_n$&$\delta_n$&$\rho_n$&$h_n$&$\#\mathcal{D}_n$&$\delta_n$&$\rho_n$
			\\\hline
			$\Sigma_1$& 0.023 & 27,903 & 0.052 & 0.002 & 0.016 & 64,737 & 0.036 & 0.0009 & 0.011 & 151,353 & 0.025 & 0.0004
			\\$\Sigma_2$& 0.023 & 4,880 & 0.052 & 0.002 & 0.016 & 11,645 &  0.036 & 0.0009 & 0.011 & 28,464 & 0.025 & 0.0004
			\\$\Sigma_3$& 0.023 & 2,315 & 0.052 & 0.002 & 0.016 & 5,517 &  0.036 & 0.0009 & 0.011 & 13,359 &0.025 & 0.0004
			\\\hline \hline
		\end{tabular}
		\normalsize
	\end{flushleft}
	\caption{\textit{Top}: Black dots illustrate global point sets for different mesh sizes. The blue dots indicate the actually used points contained in $\tilde{X}'_1$, $\tilde{X}'_2$, $\tilde{X}'_3$ (\textit{left} to \textit{right}) for the choice $\Sigma^c=\mathcal{C}_{1}(\mathbf{e}_3)$. The red line indicates the boundary of $\Sigma^c$. \textit{Bottom}: table containing information on the mesh widths $h_n$, number of basis functions in $\mathcal{D}_n$, and localization parameters $\delta_n$, $\rho_n$ for the three different choices $\Sigma^c=\mathcal{C}_{0.1}(\mathbf{e}_3),\mathcal{C}_{0.2}(\mathbf{e}_3),\,\mathcal{C}_{1}(\mathbf{e}_3)$, which we abbreviated by $\Sigma_1$, $\Sigma_2$, $\Sigma_3$ in the table.}\label{fig:points}
\end{figure}
\begin{figure}
	\begin{center}
		\includegraphics[scale=0.25]{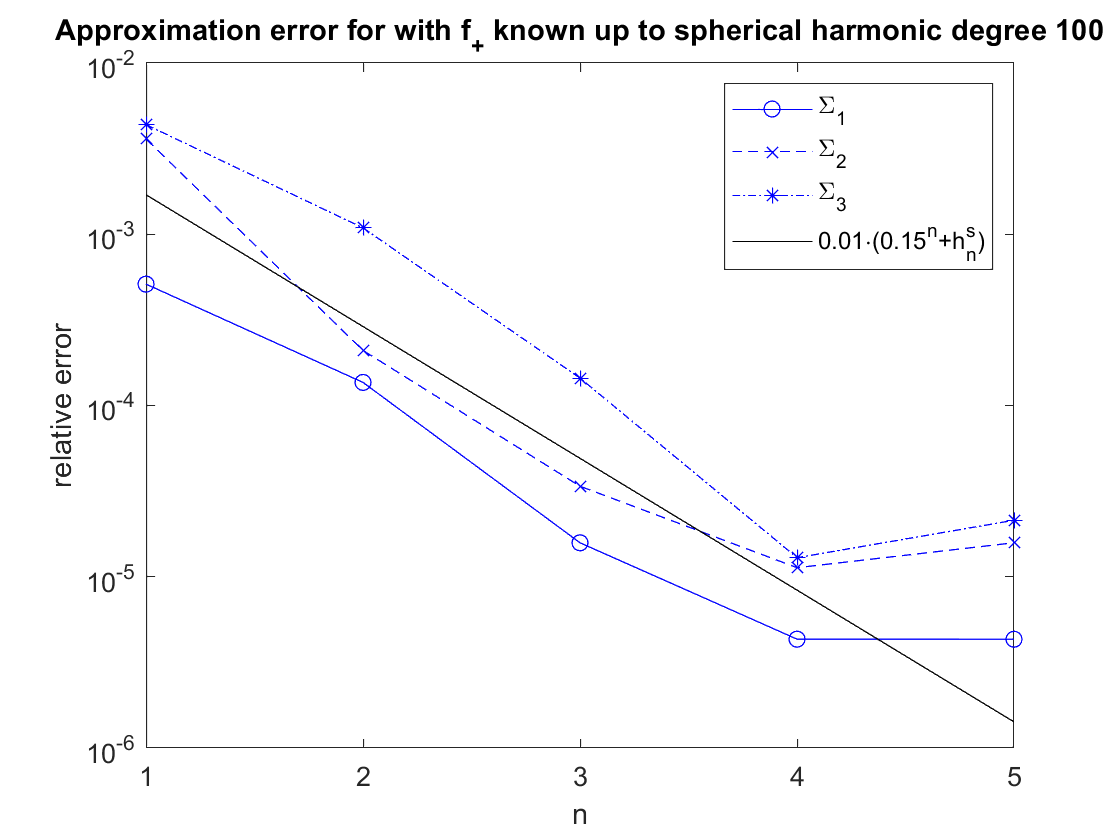}\quad\includegraphics[scale=0.25]{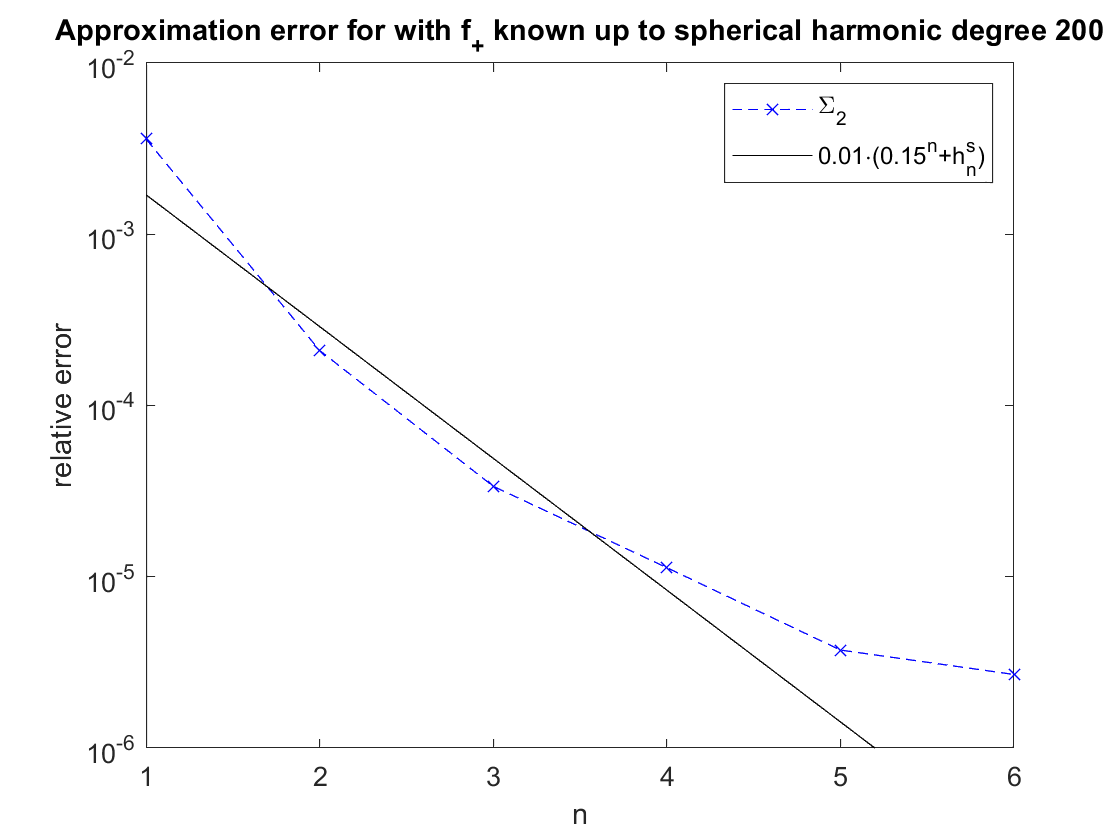}
	\end{center}
	\caption{Relative approximation error $\|f_n-f\|_{L^2(\SS)}/(\|f\|_{H^s(\SS)}+\|\tau_{ptm}(f)\|_{H^s(\SS)})$ for the $f_n$ obtained by solving \eqref{eqn:optprob} for different scales $n$ and for different choices of $\Sigma^c$, with $f$ as in \eqref{eqn:f}. The choices $\Sigma^c=\mathcal{C}_{0.1}(\mathbf{e}_3),\mathcal{C}_{0.2}(\mathbf{e}_3),\,\mathcal{C}_{1}(\mathbf{e}_3)$ are abbreviated by $\Sigma_1$, $\Sigma_2$, $\Sigma_3$ in the plots above. The black lines indicate the predicted error estimate $C(\eta^n+h_n^s)$, with the choice $C=0.01$ and $\eta=0.15$. \textit{Left}: results obtained by using information on $f$ up to spherical harmonic degree 100, \textit{right}: results when using information up to degree 200.}\label{fig:errors}
\end{figure}

\section{Conclusion}
We have investigated a set of spherical basis functions suitable for the approximation in subspaces of the Hardy space $H_+(\SS)$ that are obtained by orthogonal projection of locally supported vector fields. The new aspect has been that the considered spherical basis functions lead to vectorial functions that are themselves members of these subspaces. In this sense they are related to certain vector spherical Slepian functions but may have some advantages since they do not require simultaneous computations in $H_+(\SS)$, $H_-(\SS)$, $H_{df}(\SS)$ and only their centers need to be adapted for new domains of support. However, they do not feature an additional optimization in spectral domain that Slepian functions typically have. The obtained theoretical approximation results have been illustrated by a numerical example in Section \ref{sec:numerics}. Remark \ref{rem:ptmrbf} shows that the derived spherical basis functions further allow a simple mapping into $H_-(\SS)$, given the underlying localization constraints, which can be of interest for various inverse magnetization problems. The latter, however, requires further studies on possible regularization strategies that go beyond the scope of the paper at hand (a first na\"ive estimate is provided in the Appendix \ref{app:bep}).

\bibliography{biblio}
\bibliographystyle{plain}

\appendix

\section{Appendix}

\subsection{Proof of Proposition \ref{thm:smoothnessofconvolution}}\label{sec:prfprop21}

\begin{proof}
	If $f$ is a $\mathbf{p}-$zonal function, then 
	\begin{align}
		f(x)=\sum_{n=0}^{\infty}\hat{f}_n\sum_{k=-n}^{n}\ylk(p)\ylk(x),
	\end{align}
	for coefficients $\hat{f}_n\in\R$. From the addition theorem for spherical harmonics, we get
	\begin{align}
		\|f\|_{\sobs{s}}^2& =\sum_{n=0}^{\infty}(n+\tfrac{1}{2})^{2s}\hat{f}_n^2\sum_{k=-n}^{n}\ynk(p)^2\nonumber\\
		&=\sum_{n=0}^{\infty}(n+\tfrac{1}{2})^{2s}\hat{f}_n^2 \frac{2n+1}{4\pi}.
	\end{align}	
	Thus, it must hold
	\begin{align}
		|\hat{f}_n|\leq \sqrt{2\pi} \|f\|_{\sobs{s}} (n+\tfrac{1}{2})^{-s-\nicefrac{1}{2}}.\label{eqn:anest}
	\end{align}	
	Let $g\in \sobs{k}$ with spherical harmonic coefficient $\hat{g}_{n,k}$. The convolution defined in \eqref{eqn:conv} then has the spectral expression
	\begin{align}
		(f\ast g)^\wedge_{n,k}=\hat{f}_n \hat{g}_{n,k}
	\end{align}
	Hence we get by use of \eqref{eqn:anest} that
	\begin{align}
		\|f\ast g\|_{\sobs{s+k+\nicefrac{1}{2}}}^2&=\sum_{n=0}^\infty\sum_{k=-n}^n\hat{f}_n^2\hat{g}_{n,k}^2(n+\tfrac{1}{2})^{2s+2k+1}\nonumber
		\\&\leq\sum_{n=0}^\infty\sum_{k=-n}^n\frac{2\pi\|f\|_{H^s(\SS)}^2}{(n+\frac{1}{2})^{2s+1}}\hat{g}_{n,k}^2(n+\tfrac{1}{2})^{2s+2k+1}\nonumber
		\\&=2\pi\|f\|_{H^s(\SS)}^2\sum_{n=0}^\infty\sum_{k=-n}^n\hat{g}_{n,k}^2(n+\tfrac{1}{2})^{2k}
		\\&=2\pi  \| f\|_{\sobs{s}}^2 \| g\|_{\sobs{k}}^2,\nonumber
	\end{align}
	which concludes the proof.
\end{proof}

\subsection{Proof of Proposition \ref{Linftyboundedness}}\label{app:cinftest}

\begin{proof}
	Consider a set $\tilde{X}\subset\mathcal{B}_r=\{x\in\R^3:|x|<r\}$, with $r=1+c\nu q_X$, such that $\tilde{X}|_\SS=X$, $q_{\tilde{X}}=q_X$, $\textnormal{dist}(\tilde{X}\setminus X,\SS)>c\nu q_X$, and $h_{\tilde{X}}\leq 4c\nu q_X$. For $f$ in $H^s(\SS)$, we denote by $\tilde{f}$ an extension to $H^s(\mathcal{B}_r)$ with $\tilde{f}|_\SS=f$ and $\|\tilde{f}\|_{L^\infty(\mathcal{B}_r)}\leq 2\|f\|_{L^\infty(\SS)}$. Finally, let $\tilde{\Psi}_\delta$ be the same as $\Psi_\delta$, simply extended to act on $\mathcal{B}_r\times \mathcal{B}_r$ instead of $\SS\times\SS$ (which is unproblematic since $\Psi$ relies on a Euclidean radial basis function $\psi$ as indicated in Definition \ref{def:dilatioylkernels}). With this setup, \cite[Theorem 3.5]{towwen13} states that there exists a $C>0$, depending on $c$, $\nu$, and $s$, such that for every scale $\nu q_X<\delta< c\nu q_X$, it holds 
	\begin{equation}
		\|I_{\tilde{X}, \tilde{\Psi}_{\delta}}(\tilde{f}) \|_{L^{\infty}(\mathcal{B}_r)} \leq C \|\tilde{f} \|_{L^{\infty}(\mathcal{B}_r)}.
	\end{equation}
	The assumptions $\textnormal{dist}(\tilde{X}\setminus X,\SS)>c\nu q_X$ and $\delta<c\nu q_X$ yield that $I_{\tilde{X}, \tilde{\Psi}_{\delta}}(\tilde{f})|_\SS=I_{X, \Psi_{\delta}}(f)$ and subsequently, since by construction 
	\begin{equation}
		\|I_{X, {\Psi}_{\delta}}(f) \|_{L^{\infty}(\SS)}\leq\|I_{\tilde{X}, \tilde{\Psi}_{\delta}}(\tilde{f}) \|_{L^{\infty}(\mathcal{B}_r)} \leq C \|\tilde{f} \|_{L^{\infty}(\mathcal{B}_r)} \leq 2C \|f \|_{L^{\infty}(\SS)},
	\end{equation}
	that \eqref{eqn:linftyestsph} holds true.
\end{proof}

\subsection{A Bounded Extremal Problem for Approximation in $D_{\pm,\Sigma}$}\label{app:bep}
We are interested in the following bounded extremal problem: given $f_e=f_++e$ in $L^2(\SS)$ and $c> 0$, find $f_+^{n,e}$ in $\mathcal{D}_n$ such that 
\begin{align}\label{eqn:dbep}
f_+^{n,e}=\textnormal{argmin}\{\|f_e-g\|_{L^2(\SS)}:g\in \mathcal{D}_n,\,\|\tau_{ptm}(g)\|_{L^2(\SS)}\leq c\}.\tag{BEP}
\end{align}
The finite-dimensional subspace $\mathcal{D}_n\subset D_{+,\Sigma}$ be defined as in \eqref{eqn:choicedn}. In order to quantify the weak convergence of $\tau_{ptm}(f_+^{n,e})$ to $\tau_{ptm}(f_+)$, we further consider the following adjoint bounded extremal problem: For some $y\in L^2(\SS)$ and some bound $t>0$, find $h_t\in\textnormal{dom}(\tau_{ptm}^*)$ such that $h_t=\textnormal{argmin}\{\|y-h\|_{L^2(\SS)}:h\in \textnormal{dom}(\tau_{ptm}^*),\,\|\tau_{ptm}^*(h)\|_{L^2(\SS)}\leq t\}$. For convenience, we denote 
\begin{align}
	J_y(t)=\|y-h_t\|_{L^2(\SS)}.
\end{align}
The existence of a solution $h_t$ and the guarantee that $J_y(t)$ tends to zero as $t\to\infty$ follows along the same lines as for the original problem (e.g., \cite[App. A.4]{gerhuakeg23}). With this notation, we can now deduce the following estimate.

\begin{prop}\label{prop:weakconvergenceofdbep}
	Let $f_e=f_++e$ be in $L^2(\SS)$, with $f_+$ in $D_{+,\Sigma}$ and $\|e\|_{L^2(\SS)}<\eps$. Let all further setup be as in Theorem \ref{thm:densedsigma+} and set $\delta_n=C \left(\eta^{n}+h_{X_n}^{\kappa}\right)\left(\|f_+\|_{\sobs{s}}+\|\ptm (f_+)\|_{\sobs{s}}\right)$, with constants $C>0$ and $\eta\in(0,1)$ as indicated in Theorem \ref{thm:densedsigma+}. If $f_+^{n,e}$ is the solution to \eqref{eqn:dbep}, then, for any $y\in L^{2}(\SS)$ and $c> \|\tau_{ptm}(f_+)\|_{L^2(\SS)}$, we have that 
	\begin{align}
		\|f_+^{n,e}-f_+\|_{L^2(\SS)}\leq 2\eps+\delta_n
	\end{align} 
	and
	\begin{align}
		\left|\langle \ptm(f_+^{n,e}), y\rangle_{L^2(\SS)}-\langle \ptm(f_+), y\rangle_{L^2(\SS)}\right| \leq \inf_{t>0}\left\{2c J_{y}(t) +\Ctwo (2\eps+\delta_n) \right\}.		
	\end{align}
\end{prop}

\begin{proof}
	Without loss of generality we assume $c\geq \|\ptm(f)\|+\delta_n$. Theorem \ref{thm:densedsigma+} then guarantee the existence of an $\mathfrak{f}_n$ in $\mathcal{D}_n$ such that $\|f_+-\mathfrak{f}_n\|\leq \delta_n$ and $\|\ptm(\mathfrak{f}_n)\|\leq c$, i.e., $\mathfrak{f}_n$ lies in the feasible set of \eqref{eqn:dbep}. Thus,
	\begin{align*}
		\|f_+^{n,e}-f_+\|_{L^2(\SS)}&\leq \|f_+^{n,e}-f_e\|_{L^2(\SS)}+\eps \leq  \|\mathfrak{f}_n-f_e\|_{L^2(\SS)}+\eps
		\\&\leq\|\mathfrak{f}_n-f_+\|_{L^2(\SS)}+\|f_+-f_e\|_{L^2(\SS)}+\eps\leq \delta_n+2\eps.
	\end{align*}
	We then get that, for any $t>0$ and any $y$ in $L^2(\SS)$,
	\begin{align*}
		&\left|\langle \ptm(f_+^{n,e}), \dualelement\rangle_{L^2(\SS)}-\langle \ptm(f_+), \dualelement\rangle_{L^2(\SS)}\right| \\
		&\leq\left|\langle \ptm(f_+^{n,e})- \ptm(f_+), \dualelement- h_t\rangle_{L^2(\SS)}\right|+\left|\langle \ptm(f_+^{n,e}-f_+), h_t\rangle_{L^2(\SS)}\right|\\ 
		&\leq \left(\|\ptm(f_+^{n,e})\|_{L^2(\SS)}+\|\ptm(f_+)\|_{L^2(\SS)}\right) J_{\dualelement}(\Ctwo) + \left|\langle f_+^{n,e}-f_+, \ptm^*(h_t)\rangle_{L^2(\SS)}\right|\\
		&\leq 2cJ_{\dualelement}(\Ctwo)+(2\eps+\delta_n)t.
	\end{align*}			
	Taking the infimum over $t$ leads to the desired assertion.
\end{proof}

The proposition above guarantees that $f_+^{n,e}$ converges to $f_+$ and that $\tau_{ptm}(f_+^{n,e})$ weakly converges to $\tau_{ptm}(f_+)$ as $n\to\infty$ and $\eps\to0$ (assuming that both $f_+$ and $\tau_{ptm}(f_+)$ are in $H^s(\SS)$). The rate of the weak convergence clearly depends on the behaviour of $J_y(t)$ with respect to $t$. A more precise study of this, however, is beyond the scope of the paper at hand.

\end{document}